\documentclass[twoside,11pt]{article}

\usepackage{algorithm,algorithmic}

\newtheorem{thm}{Theorem}
\newtheorem{prop}{Proposition}

\newtheorem{ass}[thm]{Assumption}

\usepackage{enumerate}
\usepackage{booktabs}       
\usepackage{amsmath,bm,paralist,color}
\usepackage{epsf}

\usepackage{jmlr2e}
\usepackage{graphicx}
\usepackage{subfigure}

\usepackage{hyperref}

\def \x {\mathbf{x}}
\def \g {\mathbf{g}}

\def \u {\mathbf{u}}

\def \w {\mathbf{w}}

\def \R {\mathbb{R}}
\def \S {\mathcal{S}}

\def \v {\mathbf{v}}

\def \G {\mathcal {G}}

\def \wh {\widehat{\w}}

\def \s {\mathbf{s}}

\def \V {\mathcal{L}}

\usepackage{scrextend}

\jmlrheading{1}{2016}{1-37}{9/15}{}{Tianbao Yang and Qihang Lin}

\ShortHeadings{RSG: Beating Subgradient Method without Smoothness and Strong Convexity}{Yang, Lin}
\firstpageno{1}

\begin{document}

\title{RSG: Beating Subgradient Method without Smoothness and Strong Convexity}

\author{\name Tianbao Yang \email tianbao-yang@uiowa.edu\\
       \addr Department of Computer Science\\
       The University of Iowa, Iowa City, IA 52242, USA
       \AND
       \name Qihang Lin \email qihang-lin@uiowa.edu\\
       \addr Department of Management Science\\
       The University of Iowa, Iowa City, IA 52242, USA
       }

\editor{}

\maketitle
%
%
\begin{abstract}
In this paper, we study the efficiency of a {\bf R}estarted  {\bf S}ub{\bf G}radient (RSG) method that periodically restarts the standard subgradient method (SG). We show that, when applied to a broad class of convex optimization problems, RSG method can find an $\epsilon$-optimal solution with a lower complexity than the SG method. In particular, we first show that RSG can reduce the dependence of SG's iteration complexity on the distance between the initial solution and the optimal set to that between the $\epsilon$-level set and the optimal set {multiplied by a logarithmic factor}. Moreover, we show the advantages of RSG over SG in solving three different families of convex optimization problems. (a) For the problems whose epigraph is a polyhedron, RSG is shown to converge linearly. (b) For the problems with local quadratic growth property in the $\epsilon$-sublevel set, RSG has an $O(\frac{1}{\epsilon}\log(\frac{1}{\epsilon}))$ iteration complexity. (c) For the problems that admit a local Kurdyka-\L ojasiewicz property with a power constant of $\beta\in[0,1)$, RSG has an $O(\frac{1}{\epsilon^{2\beta}}\log(\frac{1}{\epsilon}))$ iteration complexity. 
The novelty of our analysis lies at exploiting the lower bound of the first-order optimality residual at the $\epsilon$-level set. It is this novelty that allows us to explore the local properties of functions (e.g., local quadratic growth property, local Kurdyka-\L ojasiewicz property, more generally local error bound conditions) to develop the improved convergence of RSG. { We also develop a practical variant of RSG  enjoying faster convergence than the SG method, which can be run without knowing the involved parameters in the local error bound condition.}  We demonstrate the effectiveness of the proposed algorithms on several machine learning tasks including regression, classification and matrix completion.
\end{abstract}


\section{Introduction}
 We consider the following generic optimization problem
\begin{equation}\label{eqn:prob0}
f_*:=\min_{\w\in\Omega} f(\w)
\end{equation}
where $f:\mathbb{R}^d\rightarrow(-\infty,+\infty]$ is an extended-valued, lower semicontinuous and convex function, and $\Omega\subseteq\R^d$ is a closed convex set such that $\Omega\subseteq\text{dom}(f)$. Here, we do not assume the smoothness of $f$ on $\text{dom}(f)$. During the past several decades, many fast (especially linearly convergent) optimization algorithms have been developed for \eqref{eqn:prob0} when $f$ is smooth and/or strongly convex. On the contrary, there are relatively fewer techniques for solving generic non-smooth and non-strongly convex optimization problems which still have many applications in machine learning, statistics, computer vision, and etc.  To solve \eqref{eqn:prob0} with $f$ being potentially non-smooth and non-strongly convex, one of the simplest algorithms to use is the subgradient (SG)~\footnote{In this paper, we use SG to refer deterministic subgradient method, though it is used in literature for stochastic gradient methods.} method. When $f$ is Lipschitz-continuous, it is known that SG method requires $O(1/\epsilon^2)$ iterations for obtaining an $\epsilon$-optimal solution~\citep{rockafellar1970convex,opac-b1104789}. It has been shown that this iteration complexity is unimprovable for general non-smooth and non-strongly convex problems in a black-box first-order oracle model of
computation~\citep{opac-b1091338}. However, better iteration complexity can be achieved by other first-order algorithms for certain classes of $f$ where additional structural information is available~\citep{Nesterov:2005:SMN,DBLP:journals/mp/GilpinPS12,2015arXivRobert,2014arXivJames,2015arXivJames,2016arXivJames}.

In this paper, we present a generic restarted subgradient (RSG) method for solving \eqref{eqn:prob0} which runs in multiple stages with each stage warm-started by the solution from the previous stage. Within each stage, the standard projected subgradient descent is performed for a fixed number of iterations with a constant step size. This step size is reduced geometrically from stage to stage. With these schemes, we show that RSG can achieve a lower iteration complexity than the classical SG method when $f$ belongs to some classes of functions. In particular, we summarize the main results and properties of RSG below:
\begin{itemize}
\item For the general problem \eqref{eqn:prob0}, under mild assumptions (see Assumption \ref{ass:rsg}), RSG has an iteration complexity of $O(\frac{1}{\epsilon^2}\log(\frac{\epsilon_0}{\epsilon}))$ which has an addition $\log(\frac{\epsilon_0}{\epsilon})$~\footnote{$\epsilon_0$ is a known upper bound of the initial optimality gap in terms of the objective value.} term but  has significantly smaller constant in $O(\cdot)$ compared to SG. {In particular, compared with SG whose iteration complexity linearly depends on the distance from the initial solution to the optimal set, RSG's iteration complexity only has a linear dependence on the distance from the $\epsilon$-level set to the optimal set, which is much smaller than the distance from the initial solution
to the optimal set. Its dependence on the initial solution is through $\epsilon_0$ - a known upper bound of the initial optimality gap, which only scales logarithmically.
}

\item When $f$ is locally quadratically growing (see Definition~\ref{def:ssc}), which is a weaker condition than strong convexity, RSG can achieve an $O(\frac{1}{\epsilon}\log(\frac{1}{\epsilon}))$ iteration complexity.
\item When $f$ admits a local Kurdyka-\L ojasiewicz property (see Definition~\ref{def:KL}) with a power desingularizing function of degree $1-\beta$ where $\beta\in[0,1)$, RSG can achieve an $O(\frac{1}{\epsilon^{2\beta}}\log(\frac{1}{\epsilon}))$ complexity.
\item When the epigraph of $f$ over $\Omega$ is a polyhedron, RSG can achieve linear convergence, i.e., an $O(\log(\frac{1}{\epsilon}))$ iteration complexity.
\end{itemize}
These results, except for the first one, are derived from a generic complexity of RSG for the problem satisfying a \emph{local error condition}~\eqref{eqn:ler}, which has a close connection to the existing error bound conditions and growth conditions~\citep{Pang:1997,Pang:1987:PEB:35577.35584,Luo:1993,DBLP:journals/corr/nesterov16linearnon,Bolte:2006:LIN:1328019.1328299} in the literature. In spite of its simplicity, the analysis of RSG provides additional insight on improving first-order methods' iteration complexity via restarting. It is known that restarting can improve the theoretical complexity of (stochastic) SG method for non-smooth problems when strongly convexity is assumed~\citep{DBLP:journals/siamjo/GhadimiL13,NIPS2012_4543,hazan-20110-beyond} but we show that restarting can be still helpful for SG methods under other (weaker) assumptions. We would like to remark that the key lemma (Lemma~\ref{lem:lowf}) developed in this work  can be leveraged to develop faster algorithms in different contexts. {For example, built on the groundwork laid in this paper
, \citet{DBLP:journals/corr/abs-1607-03815} have developed new smoothing algorithms to improve  the convergence of Nesterov's smoothing algorithm~\citep{Nesterov:2005:SMN} for non-smooth optimization with a special structure, and  \citet{DBLP:journals/corr/abs-1607-01027} have developed new stochastic gradient methods to improve the convergence  of standard stochastic subgradient method.}

{We organize the reminder of the paper as follows. Section~\ref{sec:rw} reviews some related work. Section~\ref{sec:main} presents some preliminaries and notations. Section~\ref{sec:RSG} presents the algorithm of RSG and the general theory of convergence. Section~\ref{sec:spec} considers  several classes of non-smooth and non-strongly convex problems and shows the improved iteration complexities of RSG. 
Section~\ref{sec:prac} presents parameter-free variants of RSG. Section~\ref{sec:exp} presents some experimental results. Finally, we conclude in Section~\ref{sec:conc}. }


\section{Related Work}\label{sec:rw}
Smoothness and strong convexity are two key properties of a convex optimization problem that affect the iteration complexity of finding an $\epsilon$-optimal solution by first-order methods. In general, a lower iteration complexity is expected when the problem is either smooth or strongly convex. We refer the reader to \citep{opac-b1104789,opac-b1091338} for the optimal iteration complexity of first-order methods when applied to the problems with different properties of smoothness and convexity. 
Recently there has emerged a surge of interest in further accelerating first-order methods for non-strongly convex or non-smooth problems that satisfy some particular conditions~\citep{DBLP:conf/nips/BachM13,DBLP:journals/jmlr/WangL14,DBLP:journals/corr/So13, DBLP:conf/nips/HouZSL13, DBLP:conf/icml/ZhouZS15, DBLP:journals/corr/GongY14,DBLP:journals/mp/GilpinPS12,2015arXivRobert}.
The key condition for us to develop an improved complexity is a local error bound condition~\eqref{eqn:ler} which is
closely related to the error bound conditions in~\citep{Pang:1987:PEB:35577.35584,Pang:1997,Luo:1993,DBLP:journals/corr/nesterov16linearnon,Bolte:2006:LIN:1328019.1328299,HuiZhang16a}. 

Various error bound conditions have been exploited in many studies to analyze the convergence of optimization algorithms. For example, \cite{Luo:1992a,Luo:1992b,Luo:1993} established the asymptotic linear convergence of a class of feasible descent algorithms for smooth optimization, including coordinate descent method and projected gradient method, based on a local error bound condition. Their results on coordinate descent method were further extended for a more general class of objective functions and constraints in~\citep{TsengYun:2009,TsengYun:2010}. \cite{DBLP:journals/jmlr/WangL14} showed that a global error bound holds for a family of non-strongly convex and smooth objective functions for which feasible descent methods can achieve a global linear convergence rate. Recently, these error bounds have been generalized and leveraged to show faster convergence for structured convex optimization that consists of a smooth function and a simple non-smooth function~\citep{DBLP:conf/nips/HouZSL13,ZhouSo15,DBLP:conf/icml/ZhouZS15}.  {Recently, \cite{doi:10.1137/130950288} considered a generalized error bound condition, and established linear convergence of a parallel version of a randomized (block) coordinate descent method for minimizing the sum of a partially separable smooth convex function and a fully separable non-smooth convex function.} We would like to emphasize that the aforementioned  error bounds are different from the local error bound explored in this paper. In particular, they bound the distance of  a point  to the optimal  set by using the norm of the projected gradient or proximal gradient at the point, thus requiring the smoothness of the objective function. In contrast, we bound the distance of a point to the optimal set by its objective residual with respect to the optimal value, covering a much broader family of functions.  More recently, there have appeared many studies that consider   smooth optimization or composite smooth optimization problems whose objective functions satisfy different error bound conditions, growth conditions or other non-degeneracy conditions and established the linear convergence rates of several first-order methods including proximal-gradient method, accelerated gradient method, prox-linear method and so on~\citep{DBLP:journals/corr/GongY14,DBLP:journals/corr/nesterov16linearnon,HuiZhang:2015,HuiZhang16a,DBLP:conf/pkdd/KarimiNS16,Drusvyatskiy16a,Drusvyatskiy16b,DBLP:conf/nips/HouZSL13,DBLP:conf/icml/ZhouZS15}. The relative strength and relationships between some of those conditions are studied by \cite{DBLP:journals/corr/nesterov16linearnon} and \cite{HuiZhang16a}. For example, \cite{DBLP:journals/corr/nesterov16linearnon} showed that under the smoothness assumption the second-order growth condition {(i.e., the considered error bound condition in the present work with $\theta=1/2$)} is equivalent to the error bound condition in~\citep{DBLP:journals/jmlr/WangL14}. It was brought to our attention that the local error bound condition in the present paper is closely related to metric subregularity of subdifferentials~\citep{artacho:2008,kruger2015,Drusvyatskiy14,RePEc:spr:jglopt:v:63:y:2015:i:4:p:777-795}. 

\cite{DBLP:journals/mp/GilpinPS12} established a polyhedral error bound condition for {problems whose epigraph is polyhedral and domain is a bounded polytope.} Using this polyhedral error bound condition, they studied a two-person zero-sum game and proposed a restarted first-order method based on Nesterov's smoothing technique~\citep{Nesterov:2005:SMN} that can find the Nash equilibrium and has linear convergence rate. 
{The differences between~\citep{DBLP:journals/mp/GilpinPS12} and this work are: (i)  we study subgradient methods instead of Nesterov's smoothing technique, where the former have broader applicability than Nesterov's smoothing technique; (ii) our linear convergence can be derived for a slightly general problem where the domain is allowed to be an unbounded polyhedron as long as the polyhedral error bound condition in Lemma~\ref{lem:polyeb} holds, which is the case for many important applications; (iii) we consider a general condition that subsumes the polyhedral error bound condition as a special case and we try to solve the general problem \eqref{eqn:prob0} rather than the bilinear saddle-point problem in~\citep{DBLP:journals/mp/GilpinPS12}. }

{The error bound condition that allows us to derive a linear convergence of RSG is the same to the weak sharp minimum condition, which was first coined  in 1970s~\citep{polysharp1979}. However, it was used  even  earlier  for studying the convergence of subgradient method~\citep{eremin1965,polyak1969}.  Later, it  was studied in many subsequent works~\citep{citeulike:13796896,doi:10.1137/0331063,doi:10.1137/S0363012996301269,Ferris1991,DBLP:journals/mp/BurkeD02,DBLP:journals/mp/BurkeD05,DBLP:journals/mp/BurkeD09}. Finite or linear convergence of several algorithms has been established under the weak sharp minimum condition, including gradient projection method~\citep{citeulike:13796896}, the proximal point algorithm (PPA)~\citep{Ferris1991}, and subgradient method with a particular choice of step size (see below)~\citep{polyak1969}.  We would like to emphasize the differences between the results in these works and the results in the present work that make our results novel: (i) the gradient projection method and its finite convergence established in~\citep{citeulike:13796896} requires the gradient of the objective function to be Lipschitz continuous, i.e., the objective function is smooth (please refer to Theorem 1 (Chapter 7, pp 207) in~\citep{citeulike:13796896}),  in contrast we do not assume smoothness of the objective function; (ii) the PPA  studied in~\citep{Ferris1991} requires solving a proximal sub-problem consisting of the original objective function and a strongly convex function at every iteration, and therefore its finite convergence does not mean that only a finite number of subgradient evaluations is needed. In contrast, the linear convergence in this paper was in terms of the number of subgradient evaluations; (iii) linear convergence of a subgradient method studied in~\citep{polyak1969} requires knowing the optimal objective value for setting its step size, and its convergence is in terms of the distance of the iterates to the optimal set, which is weaker than our linear convergence in terms of objective gap.  In addition, our method does not require knowing the optimal objective value. Instead the basic variant of RSG that has a linear convergence only needs to know the value of the multiplicative constant parameter in the local error bound condition. For problems without knowing this parameter, we also develop a practical variant of RSG that can achieve a convergence rate close to linear convergence.} 


In his recent work~\citep{2014arXivJames,2015arXivJames,2016arXivJames}, Renegar presented a framework of applying first-order methods to general conic optimization problems by transforming the original problem into an equivalent convex optimization problem with only linear equality constraints and a Lipschitz-continuous objective function. This framework greatly extends the applicability of first-order methods to the problems with general linear inequality constraints and leads to new algorithms and new iteration complexity. One of his results related to this work is Corollary 3.4 of~\citep{2015arXivJames}, which implies, if the objective function has a polyhedral epigraph and the optimal objective value is known beforehand, a subgradient method can have a linear convergence rate. Compared to this result of his, our method does not need to know the optimal objective value. 
Note that Renegar's method can be applied in a general setting where the objective function is not necessarily polyhedral while our method obtains improved iteration complexities under the local error bound conditions.


More recently, \cite{2015arXivRobert} proposed a new SG method by assuming that a strict lower bound of $f_*$, denoted by $f_{slb}$, is known and $f$ satisfies a growth condition, $\|\w-\w^*\|_2\leq \mathcal{G}\cdot(f(\w)-f_{slb})$, where $\w^*$ is the optimal solution closest to $\w$ and $\mathcal{G}$ is a growth rate constant depending on $f_{slb}$. Using a novel step size that incorporates $f_{slb}$, for non-smooth optimization, their SG method achieves an iteration complexity of $O(\mathcal{G}^2(\frac{\log H}{\epsilon'} + \frac{1}{\epsilon'^2}))$ for finding a solution $\hat{\w}$ such that $f(\hat{\w}) - f_*\leq \epsilon'(f_* - f_{slb})$, where $H = \frac{f(\w_0)-f_{slb}}{f_*-f_{slb}}$ and $\w_0$ is the initial solution. We note that there are several key differences in the theoretical properties and implementations between our work and~\citep{2015arXivRobert}: (i) Their growth condition has a similar form to the inequality~\eqref{eqn:keyii} we prove for a general function but there are still noticeable differences in the both sides and the growth constants. (ii) The convergence results in~\citep{2015arXivRobert} are established based on finding an solution $\hat{\w}$ with a relative error of $\epsilon'$ while we consider absolute error. (iii) By rewriting the convergence results in~\citep{2015arXivRobert} in terms of absolute accuracy $\epsilon$ with $\epsilon=\epsilon'(f_* - f_{slb})$, the complexity in~\citep{2015arXivRobert} depends on $f_*-f_{slb}$ and may be higher than ours if $f_*-f_{slb}$ is large. 
{However, Freund and Lu's new SG method is still attractive due to that it is a parameter free algorithm without requiring the value of the growth constant $\mathcal G$.} We will compared our RSG method with the method in~\citep{2015arXivRobert} with more details in Section~\ref{sec:diss}.

Restarting and multi-stage strategies have been utilized to achieve the (uniformly) optimal theoretical complexity of (stochastic) SG methods when $f$ is strongly convex~\citep{DBLP:journals/siamjo/GhadimiL13,NIPS2012_4543,hazan-20110-beyond} {or uniformly convex~\citep{Nesterov:2014:uniform_convex}}. Here, we show that restarting can be still helpful even without {uniform or} strong convexity. Furthermore, in all the algorithms proposed in \citep{DBLP:journals/siamjo/GhadimiL13,NIPS2012_4543,hazan-20110-beyond,Nesterov:2014:uniform_convex}, the number of iterations per stage increases between stages while our algorithm uses the same number of iterations in all stages. This provides a different possibility of designing restarted algorithms for a better complexity only under a local error bound condition.

\section{Preliminaries}\label{sec:main}
In this section, we define some notations used in this paper and present the main assumptions needed to establish our results.
We use $\partial f(\w)$ to denote the set of subgradients (the  subdifferential) of $f$ at $\w$.  
Since the objective function is not necessarily strongly convex, the optimal solution is not necessarily unique. We denote by $\Omega_*$  the optimal solution set and by $f_*$  the unique optimal objective value. We denote by $\|\cdot\|_2$ the Euclidean norm in $\mathbb{R}^d$.

Throughout the paper, we make the following assumption.
\begin{ass}\label{ass:rsg} For the convex minimization problem~(\ref{eqn:prob0}), we assume
\begin{itemize}
\item[\bf a.] For any $\w_0\in\Omega$, we know a constant $\epsilon_0\geq 0$ such that $f(\w_0) - f_*\leq \epsilon_0$.
\item[\bf b.] There exists a constant $G$ such that $\max_{\v\in\partial f(\w)}\|\v\|_2\leq G$  for any $\w\in\Omega$.
\end{itemize}
\end{ass}
We make several remarks about the above assumptions:  (i) Assumption~\ref{ass:rsg}.a is equivalent to assuming we know
a lower bound of $f_*$ which is one of the assumptions made in \citep{2015arXivRobert}. In machine learning applications, $f_*$ is usually bounded below by zero, i.e., $f_*\geq 0$, so that $\epsilon_0 = f(\w_0)$ for any $\w_0\in\R^d$ will satisfy the condition; (ii) Assumption~\ref{ass:rsg}.b is a standard assumption also made in many previous subgradient-based methods. 

Let $\w^*$ denote the closest optimal solution in $\Omega_*$ to $\w$ measured in terms of norm $\|\cdot\|_2$, i.e.,
\[
\w^* := \arg\min_{\u\in\Omega_*}\|\u - \w\|^2_2.
\]
Note that $\w^*$ is uniquely defined for any $\w$ due to the convexity of $\Omega_*$ and that $\|\cdot\|_2^2$ is strongly convex.
We denote by  $\V_\epsilon$ the $\epsilon$-level set of $f(\w)$ and  by $\S_\epsilon$ the $\epsilon$-sublevel set of $f(\w)$, respectively, i.e.,
\begin{eqnarray}
\label{def:LeSe}
\V_\epsilon := \{\w\in\Omega: f(\w) = f_* + \epsilon\}\quad\text{and}\quad\S_\epsilon := \{\w\in\Omega: f(\w) \leq f_* + \epsilon\}.
\end{eqnarray}
Let $B_\epsilon$ be the maximum distance between the points in the $\epsilon$-level set $\V_\epsilon$ and the optimal set $\Omega_*$, i.e.,
\begin{equation}\label{eqn:keyB}
B_{\epsilon}:= \max_{\w\in\V_\epsilon}\min_{\u\in\Omega_*}\|\w - \u\|_2 = \max_{\w\in\V_\epsilon}\|\w - \w^*\|_2.
\end{equation}
In the sequel, we also make the following assumption. 
{\begin{ass}\label{ass:rsg2} For the convex minimization problem~(\ref{eqn:prob0}), we assume that $B_\epsilon$ is finite. 
\end{ass}
{\bf Remark:} $B_\epsilon$ is finite when the optimal set $\Omega_*$ is bounded (e.g.,  when the objective function is a proper lower-semicontinuous convex and coercive function). This is because that  following~\citep{rockafellar1970convex} (Corollary 8.7.1)  the sublevel set $\S_\epsilon$ must be bounded for any $\epsilon\geq 0$. Nevertheless, the bounded optimal set is not a necessary condition for a finite $B_\epsilon$. For example, $f(x)=\max(0, x)$. Although its optimal set is not bounded,  $B_\epsilon=\epsilon$.  In Section~\ref{sec:spec}, we will consider a broad family of problems with a local error bound condition, which will satisfy  the above assumption.}

Let $\w_{\epsilon}^\dagger$ denote the closest point in the $\epsilon$-sublevel set to $\w$, i.e.,
\begin{equation}
\label{eq:wproj}
\begin{aligned}
\w_\epsilon^{\dagger}:=\arg\min_{\u\in\S_\epsilon}\|\u - \w\|^2_2
\end{aligned}
\end{equation}
Denote by $\Omega\backslash\S=\{\w\in\Omega: \w\not\in\S\}$.  It is easy to show that $\w_\epsilon^{\dagger}\in\V_\epsilon$ when $\w\in\Omega\backslash\S_\epsilon$ (using the optimality condition of \eqref{eq:wproj}).

Given $\w\in\Omega$, we denote the normal cone of $\Omega$ at $\w$ by $\mathcal{N}_{\Omega}(\w)$.  Formally, $\mathcal N_\Omega(\w)= \{\v\in\R^d: \v^{\top}(\u - \w)\leq 0, \forall \u\in\Omega\}$. Define
$\text{dist}(0, f(\w)+\mathcal{N}_{\Omega}(\w))$ as
\begin{equation}
\label{eq:bdsubgrad}
\text{dist}(0, f(\w)+\mathcal{N}_{\Omega}(\w)):=\min_{\mathbf g\in\partial f(\w),\mathbf v\in\mathcal{N}_{\Omega}(\w)}  \|\mathbf g+\mathbf v\|_2.
\end{equation}
Note that $\w\in\Omega_*$ if and only if $\text{dist}(0, f(\w)+\mathcal{N}_{\Omega}(\w))=0$. Therefore, we call $\text{dist}(0, f(\w)+\mathcal{N}_{\Omega}(\w))$ the \emph{first-order optimality residual} of \eqref{eqn:prob0} at $\w\in\Omega$. Given any $\epsilon>0$ such that $\V_\epsilon\neq\emptyset$, we define a constant $\rho_\epsilon$ as
\begin{equation}
\label{eq:defrho}
\rho_\epsilon:=\min_{\w\in\V_\epsilon}  \text{dist}(0, f(\w)+\mathcal{N}_{\Omega}(\w)).
\end{equation}

Given the notations above, we provide the following  lemma which is the key to our analysis.
\begin{lemma}
\label{lem:1}
For any $\epsilon>0$ such that $\V_\epsilon\neq\emptyset$ and any $\w\in\Omega$, we have
\begin{equation}\label{eqn:keyii}
\begin{aligned}
\|\w - \w_\epsilon^\dagger\|_2\leq \frac{1}{\rho_\epsilon}(f(\w) - f(\w_\epsilon^\dagger)).
\end{aligned}
\end{equation}
\end{lemma}
\begin{proof}
Since the conclusion holds trivially if $\w\in\S_\epsilon$ (so that $\w^\dagger_\epsilon=\w$), we assume $\w\in\Omega\backslash\S_\epsilon$. According to the first-order optimality conditions of~(\ref{eq:wproj}), there exist a scalar $\zeta\geq 0$ (the Lagrangian multiplier of the constraint $f(\u)\leq f_* + \epsilon$ in~(\ref{eq:wproj})), a subgradient $\mathbf{g}\in \partial f(\w^\dagger_\epsilon)$ and a vector $\mathbf v\in\mathcal{N}_{\Omega}(\w^\dagger_\epsilon)$ such that
\begin{equation}\label{eqn:o2}
\begin{aligned}
&\w^\dagger_\epsilon - \w + \zeta \mathbf{g}+\mathbf v=0.
\end{aligned}
\end{equation}
The definition of normal cone leads to $
(\w_\epsilon^\dagger - \w)^{\top}\mathbf v\geq 0$.
This inequality and the convexity of $f(\cdot)$ imply
\[
\zeta\left(f(\w) - f(\w^\dagger_\epsilon)\right)\geq \zeta(\w - \w^\dagger_\epsilon)^{\top}\mathbf g
\geq (\w - \w^\dagger_\epsilon)^{\top}\left(\zeta\mathbf g+\mathbf v\right)
= \|\w- \w^\dagger_\epsilon\|_2^2
\]
where the equality is due to \eqref{eqn:o2}. Since $\w\in\Omega\backslash\S_\epsilon$, we must have $\|\w- \w^\dagger_\epsilon\|_2>0$ so that $\zeta>0$. Therefore, $\w^\dagger_\epsilon\in\V_\epsilon$ by complementary slackness.
Dividing the inequality above by $\zeta$ gives
\begin{eqnarray}
\label{eq:lemma2}
f(\w) - f(\w^\dagger_\epsilon)\geq \frac{\|\w- \w^\dagger_\epsilon\|_2^2}{\zeta}=\|\w- \w^\dagger_\epsilon\|_2\|\mathbf g+\mathbf v/\zeta \|_2\geq\rho_\epsilon\|\w- \w^\dagger_\epsilon\|_2,
\end{eqnarray}
where the equality is due to \eqref{eqn:o2} and the last inequality is due to the definition of $\rho_\epsilon$ in \eqref{eq:defrho}. The lemma is then proved.
\end{proof}

The inequality in~(\ref{eqn:keyii}) is the key to achieve improved convergence by RSG, which hinges on the condition that the first-order optimality residual on the $\epsilon$-level set is lower bounded. It is important to note that (i) the above result depends on $f$ rather than the optimization algorithm applied; and (ii) the above result can be generalized to use other norm such as the $p$-norm $\|\w\|_p$ ($p\in(1,2]$) to measure the distance between $\w$ and $\w^\dagger_\epsilon$ and use the corresponding dual norm to define the lower bound of the residual in \eqref{eq:bdsubgrad} and \eqref{eq:defrho}. This generalization allows one to design mirror decent~\citep{Nemirovski:2009:RSA:1654243.1654247} variant of RSG. 
To our best knowledge, this is the first work that leverages the lower bound of the optimal residual to improve the convergence for non-smooth convex optimization.

In the next several sections, we will exhibit the value of $\rho_\epsilon$ for different classes of problems and discuss its impact on the convergence. In the sequel, we abuse the Big O notation $T=O(h(\epsilon))$ to mean that there exists a constant $C>0$ independent of $\epsilon$ such that $T\leq Ch(\epsilon)$.

\section{Restarted SubGradient (RSG) Method and Its Generic Complexity for General Problem}\label{sec:RSG}
In this section, we present a framework of restarted subgradient (RSG) method and prove its general convergence result using Lemma~\ref{lem:1}. {It will be noticed that the algorithmic results developed in this section is less interesting from the viewpoint of practice. However, it will exhibit the insights for the improvements and provide the template for the developments in next several sections,} where we will present improved convergence of RSG for problems of different classes.

The steps of RSG are presented in Algorithm~\ref{alg:1} where SG is a subroutine of projected subgradient descent given in Algorithm~\ref{alg:0} and $\Pi_{\Omega}[\w]$ is defined as
\[
\Pi_{\Omega}[\w] = \arg\min_{\u\in\Omega}\|\u - \w\|_2^2.
\]
{The values of $K$ and $t$ in RSG will be revealed later for proving the convergence of RSG to an $2\epsilon$-optimal solution. }
The number of iterations $t$ is the only varying parameter in RSG that depends on the classes of problems. The parameter $\alpha$ could be any value larger than $1$ (e.g., 2) and it only has a small influence on the iteration complexity. 


We emphasize that (i) RSG is a generic algorithm that is applicable to a broad family of non-smooth and/or non-strongly convex problems without changing updating schemes except for one tuning parameter, the number of iterations per stage, whose best value varies with problems; (ii) RSG has different variants with different subroutines in stages. In fact,  we can use other optimization algorithms than SG as the subroutine in Algorithm~\ref{alg:1}, as long as a similar convergence result to Lemma~\ref{lem:GD} is guaranteed. Examples include dual averaging~\citep{Nesterov:2009:PSM:1530733.1530741} 
and the regularized dual averaging~\citep{NIPS2012_4543} in the non-Euclidean space.  In the following discussions, we will focus on using SG as the subroutine.
\begin{algorithm}[t]
\caption{SG:  $\wh_T = \text{SG}(\w_1, \eta, T)$} \label{alg:0}
\begin{algorithmic}[1]
\STATE \textbf{Input}: a step size $\eta$,  the  number of iterations $T$, and the initial solution $\w_1\in\Omega$
\FOR{$\tau=1,\ldots, T$}
    \STATE Query the subgradient oracle to obtain $\G(\w_\tau)\in\partial f(\w_{\tau})$
    \STATE Update $\w_{\tau+1} =\Pi_{\Omega}[ \w_\tau - \eta \G(\w_\tau)]$
   \ENDFOR
\STATE \textbf{Output}:  $\wh_T = \sum_{\tau=1}^T\frac{\w_\tau}{T}$
\end{algorithmic}
\end{algorithm}
\begin{algorithm}[t]
\caption{RSG: $\w_K=\text{RSG}(\w_0,  K, t, \alpha$)} \label{alg:1}
\begin{algorithmic}[1]
\STATE \textbf{Input}: the number of stages $K$ and the number of iterations $t$ per-stage, $\w_0 \in\Omega$, and $\alpha>1$.
\STATE Set $\eta_1=\epsilon_0/(\alpha G^2)$, where $\epsilon_0$ is from  Assumption~\ref{ass:rsg}.a
\FOR{$k=1,\ldots, K$}
    \STATE Call subroutine SG to obtain $\w_k = \text{SG}(\w_{k-1}, \eta_k, t)$
\STATE Set $\displaystyle \eta_{k+1} = \eta_{k}/\alpha$
\ENDFOR
\STATE \textbf{Output}:  $\w_K$
\end{algorithmic}
\end{algorithm}


Next, we establish the convergence of RSG. It relies on the convergence result of the SG subroutine which is given in the lemma below. 
\begin{lemma}\citep{zinkevich-2003-online,opac-b1104789} \label{lem:GD}
If Algorithm~\ref{alg:0} runs for $T$ iterations, we have, for any $\w\in\Omega$,
\[
f(\wh_T) - f(\w) \leq \frac{G^2\eta}{2} + \frac{\|\w_{1} - \w\|_2^2}{2\eta  T}.
\]
\end{lemma}
We omit the proof because it follows a standard analysis and can be found in cited papers. With the above lemma, we can prove the following convergence of RSG. 
\begin{theorem}\label{thm:GDr}
Suppose Assumption~\ref{ass:rsg} holds.  If $t\geq \frac{\alpha^2 G^2}{\rho_\epsilon^2}$ and $K = \lceil \log_\alpha(\frac{\epsilon_0}{\epsilon})\rceil$ in Algorithm~\ref{alg:1}, with at most $K$ stages, Algorithm~\ref{alg:1} returns a solution $\w_K$ such that  $f(\w_K)-f_*\leq 2\epsilon$. 
The total number of iterations for Algorithm~\ref{alg:1} to find an $2\epsilon$-optimal solution is at most $T= t\lceil \log_\alpha(\frac{\epsilon_0}{\epsilon})\rceil$ where $t\geq \frac{\alpha^2 G^2}{\rho_\epsilon^2}$. 
\end{theorem}
{\bf Remark:} If $t$ also satisfies $t=O\left( \frac{\alpha^2 G^2}{\rho_\epsilon^2}\right)$, then the iteration complexity of Algorithm~\ref{alg:1} for finding an $\epsilon$-optimal solution is  $O\left(\frac{\alpha^2G^2}{\rho_\epsilon^2}\lceil \log_\alpha(\frac{\epsilon_0}{\epsilon})\rceil\right)$.

\begin{proof}

Let $\w_{k,\epsilon}^\dagger$ denote the closest point to $\w_k$ in the $\epsilon$-sublevel set.  Let $\epsilon_k:=\frac{\epsilon_0}{\alpha^k}$ so that $\eta_k = \epsilon_k/G^2$ because $\eta_1=\epsilon_0/(\alpha G^2)$ and $\eta_{k+1} = \eta_{k}/\alpha$. We will show by induction that
\begin{equation}\label{eqn:mainineq1}
f(\w_k) - f_*\leq \epsilon_k +\epsilon
\end{equation}
for $k=0,1,\dots,K$ which leads to our conclusion if we let $k=K$.

Note that \eqref{eqn:mainineq1} holds obviously for $k=0$. Suppose it holds for $k-1$, namely, $f(\w_{k-1}) - f_*\leq \epsilon_{k-1} + \epsilon$. We want to prove \eqref{eqn:mainineq1} for $k$. We apply Lemma~\ref{lem:GD} to the $k$-th stage of Algorithm~\ref{alg:1} and get
\begin{equation}\label{eqn:mainineq2}
f(\w_k) - f(\w_{k-1, \epsilon}^\dagger)\leq \frac{G^2\eta_k}{2} + \frac{\|\w_{k-1} - \w_{k-1,\epsilon}^\dagger\|_2^2}{2\eta_k t}.
\end{equation}
We now consider two cases for $\w_{k-1}$. First, assume $f(\w_{k-1}) - f_* \leq \epsilon$, i.e., $\w_{k-1}\in\mathcal S_\epsilon$. Then $\w^\dagger_{k-1,\epsilon} = \w_{k-1}$ and $f(\w_k) - f(\w^\dagger_{k-1,\epsilon}) \leq \frac{G^2\eta_k}{2} = \frac{\epsilon_k}{2}$.
As a result,
\[
f(\w_k) - f_* \leq f(\w_{k-1,\epsilon}^\dagger) - f_* + \frac{\epsilon_k}{2} \leq  \epsilon + \epsilon_k
\]
Next, we consider the case that $f(\w_{k-1}) - f_*>\epsilon$, i.e., $\w_{k-1}\not\in\S_\epsilon$. Then we have $f(\w_{k-1,\epsilon}^\dagger) = f_* + \epsilon$.  By Lemma~\ref{lem:1}, we have
\begin{equation}\label{eqn:mainineq3}
\begin{aligned}
\|\w_{k-1} - \w^\dagger_{k-1,\epsilon}\|_2&\leq \frac{1}{\rho_\epsilon}(f(\w_{k-1}) - f(\w_{k-1,\epsilon}^\dagger))=\frac{ f(\w_{k-1}) - f_*  + (f_* - f(\w^\dagger_{k-1,\epsilon}))}{\rho_\epsilon}\\
&\leq \frac{\epsilon_{k-1}+\epsilon - \epsilon}{\rho_\epsilon}.
\end{aligned}
\end{equation}
Combining  \eqref{eqn:mainineq2} and \eqref{eqn:mainineq3} and using the facts that $\displaystyle \eta_k =\frac{\epsilon_k}{G^2}$  and $\displaystyle t \geq \frac{\alpha^2 G^2}{\rho_\epsilon^2}$, we have
\[
f(\w_k) - f(\w^\dagger_{k-1,\epsilon})\leq   \frac{\epsilon_k}{2} + \frac{\epsilon_{k-1}^2}{2\epsilon_k\alpha^2} = \epsilon_k
\]
which, together with the fact that $ f(\w^\dagger_{k-1,\epsilon})=f_*+\epsilon$, implies \eqref{eqn:mainineq1} for $k$.
Therefore, by induction, we have \eqref{eqn:mainineq1} holds for $k=1,2,\dots,K$ so that
\[
f(\w_K) - f_*\leq \epsilon_K  + \epsilon = \frac{\epsilon_0}{\alpha^K} +\epsilon\leq 2\epsilon,
\]
where the last inequality is due to the definition of $K$.
\end{proof}


In Theorem~\ref{thm:GDr}, the iteration complexity of RSG for the general problem \eqref{eqn:prob0} is given in terms of $\rho_\epsilon$. Next, we show that $\rho_\epsilon\geq\frac{\epsilon}{B_\epsilon}$, which allows us to leverage the local error bound condition in next sections to upper bound  $B_\epsilon$ to obtain specialized and more practical algorithms for different classes of problems.
\begin{lemma}\label{lem:lowf}
For any $\epsilon>0$ such that $\V_\epsilon\neq\emptyset$, we have $
\rho_\epsilon\geq \frac{\epsilon}{B_\epsilon}$,
where $B_\epsilon$ is defined in~(\ref{eqn:keyB}), and for any $\w\in\Omega$
\begin{equation}\label{eqn:keyto}
\|\w - \w_\epsilon^\dagger\|_2 \leq \frac{\|\w_\epsilon^\dagger - \w^*_\epsilon\|_2}{\epsilon}(f(\w) - f(\w_\epsilon^\dagger))\leq \frac{B_\epsilon}{\epsilon}(f(\w) - f(\w_\epsilon^\dagger)),
\end{equation}
where $\w_\epsilon^*$ is the closest point in $\Omega_*$ to $\w^\dagger_\epsilon$.
\end{lemma}

\begin{figure}
\centering
\includegraphics[scale=1.2]{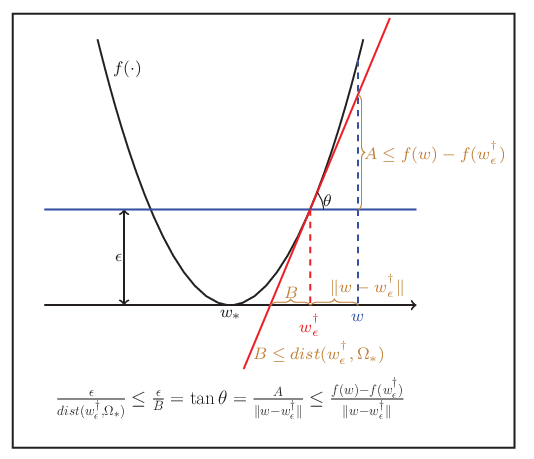}
\caption{A geometric illustration of the inequality~(\ref{eqn:keyto}), where $dist(w^\dagger_\epsilon,\Omega_*) = |w^\dagger_\epsilon  - w^*_\epsilon|$. }
\label{fig:2}
\end{figure}

\begin{proof}
Given any $\u\in\V_\epsilon$, let $\mathbf g_\u$ be any subgradient in $\partial f(\u)$ and $\mathbf v_\u$ be any vector in $\mathcal{N}_{\Omega}(\u)$. By the convexity of $f(\cdot)$ and the definition of normal cone, we have
\[
f(\u^*) - f(\u)\geq (\u^* - \u)^{\top}\mathbf g_\u
\geq(\u^* - \u)^{\top}\left(\mathbf g_\u+\mathbf v_\u\right),
\]
where $\u^*$ is the closest point in $\Omega_*$ to $\u$. This inequality further implies
\begin{equation}
\label{eq:lowboundeq1}
\begin{aligned}
 \|\u^* - \u\|_2&\|\mathbf g_\u+\mathbf v_\u\|_2\geq  f(\u) - f(\u^*) = \epsilon, \quad\forall \g_\u\in\partial f(\u) \text{ and } \v_\u\in\mathcal N_{\Omega}(\u)
 \end{aligned}
\end{equation}
where the equality is because $\u\in\V_\epsilon$. By \eqref{eq:lowboundeq1} and the definition of $B_\epsilon$, we obtain
\[
B_\epsilon\|\mathbf g_\u+\mathbf v_\u\|_2\geq \epsilon\Longrightarrow \|\mathbf g_\u+\mathbf v_\u\|_2\geq \epsilon/B_\epsilon.
\]
Since $\mathbf g_\u+\mathbf v_\u$ can be any element in $\partial f(\u)+\mathcal{N}_{\Omega}(\u)$, we have $\rho_\epsilon\geq \frac{\epsilon}{B_\epsilon}$ by the definition \eqref{eq:defrho}.

To prove~(\ref{eqn:keyto}), we assume $\w\in\Omega\backslash\S_\epsilon$ and thus $\w^\dagger_\epsilon\in\V_\epsilon$; otherwise it is trivial. In the proof of Lemma~\ref{lem:1}, we have shown that (see \eqref{eq:lemma2}) there exists $\g\in\partial f(\w_\epsilon^\dagger)$ and $\v\in\mathcal N_\Omega(\w^\dagger_\epsilon)$ such that
$
f(\w) - f(\w^\dagger_\epsilon)\geq\|\w- \w^\dagger_\epsilon\|_2\|\mathbf g+\mathbf v/\zeta \|_2
$, which, according to~(\ref{eq:lowboundeq1}) with $\u=\w^\dagger_\epsilon$, $\mathbf g_\u=\mathbf g$ and $\mathbf v_\u=\mathbf v/\zeta$, leads to (\ref{eqn:keyto}).
\end{proof}
A geometric explanation of the inequality~(\ref{eqn:keyto}) in one dimension is shown in Figure~\ref{fig:2}. With Lemma~\ref{lem:lowf}, the iteration complexity of RSG can be stated in terms of $B_\epsilon$ in the following corollary of Theorem~\ref{thm:GDr}.
\begin{corollary}\label{lem:3}
Suppose Assumption~\ref{ass:rsg} holds.  The iteration complexity of RSG for obtaining an $2\epsilon$-optimal solution is $O(\frac{\alpha^2G^2B_\epsilon^2}{\epsilon^2}\lceil \log_\alpha(\frac{\epsilon_0}{\epsilon})\rceil)$ provided $t=\frac{\alpha^2G^2B_\epsilon^2}{\epsilon^2}$ 
and $K = \lceil \log_\alpha(\frac{\epsilon_0}{\epsilon})\rceil$.
\end{corollary}
We will compare this result with SG in Section~\ref{sec:diss}. Compared to the standard SG, the above improved result of RSG does require knowing strong knolwedge about $f$. In particular, one issue is that the above improved complexity is obtained by choosing $t =\frac{\alpha^2G^2B_\epsilon^2}{\epsilon^2}$, which requires knowing the order of magnitude of $B_\epsilon$, if not its exact value. To address the issue of unknown $B_\epsilon$ for general problems, in the next section, we consider different families  of problems that admit a local error bound condition and show that the requirement of knowing  $B_\epsilon$ is relaxed to knowing some particular parameters related to the local error bound.

\section{RSG for Some Classes of Non-smooth Non-strongly Convex Optimization}\label{sec:spec}
In this section, we consider a particular family of problems that admit local error bounds and show the improved iteration complexities of RSG compared to standard SG method.

\subsection{Complexity for the Problems with Local Error Bounds}\label{sec:ler}
We first define local error bound of the objective function.
\begin{definition}
\label{def:ler}
We say $f(\cdot)$ admits a \textbf{local error bound} on the $\epsilon$-sublevel set $\S_\epsilon$ if
\begin{equation}\label{eqn:ler}
\|\w - \w^*\|_2 \leq c(f(\w) - f_*)^{\theta}, \quad \forall \w\in\S_\epsilon,
\end{equation}
where $\w^*$ is the closet point in $\Omega_*$ to $\w$, $\theta\in(0,1]$ and $0<c<\infty$ are constants.
\end{definition}
Because $\S_{\epsilon_2}\subset\S_{\epsilon_1}$ for $\epsilon_2\leq\epsilon_1$, if \eqref{eqn:ler} holds for some $\epsilon$, it will always hold when $\epsilon$ decreases to zero with the same $\theta$ and $c$. Indeed, a smaller $\epsilon$ may induce a smaller value of $c$. {It is notable that the local error bound condition has been extensively studied in the community of optimization, mathematical programming and variational analysis~\citep{doi:10.1137/070689838,DBLP:journals/siamjo/Li10,DBLP:journals/mp/Li13,artacho:2008,kruger2015,Drusvyatskiy14,DBLP:journals/siamjo/LiM12,DBLP:conf/nips/HouZSL13,ZhouSo15,DBLP:conf/icml/ZhouZS15}, to name just a few of them. The value of $\theta$ has been exhibited for many problems. For certain problems, the value of $c$ is also computable (c.f. examples in~\cite{arxiv:1510.08234}).}

If the problem admits a local error bound like~(\ref{eqn:ler}), RSG can achieve a better iteration complexity than $O(1/\epsilon^2)$. In particular, the property~(\ref{eqn:ler}) implies
\begin{equation}
\label{eqn:Bebound}
B_\epsilon\leq c\epsilon^{\theta}.
\end{equation}
Replacing $B_\epsilon$ in Corollary~\ref{lem:3} by this upper bound and choosing $t=\frac{\alpha^2G^2c^2}{\epsilon^{2(1-\theta)}}$ in RSG if $c$ and $\theta$ are known, we obtain the following complexity of RSG.
\begin{corollary}\label{lem:leb_corr}
Suppose Assumption~\ref{ass:rsg} holds and $f(\cdot)$ admits a local error bound on $\S_\epsilon$.  The iteration complexity of RSG for obtaining an $2\epsilon$-optimal solution is $O\left(\frac{\alpha^2G^2c^2}{\epsilon^{2(1-\theta)}}\log_\alpha\left(\frac{\epsilon_0}{\epsilon}\right)\right)$ provided $t =\frac{\alpha^2G^2c^2}{\epsilon^{2(1-\theta)}}$ and $K = \lceil \log_\alpha(\frac{\epsilon_0}{\epsilon})\rceil$.
\end{corollary}
{\bf Remark:} If $t=\Theta(\frac{\alpha^2G^2}{\epsilon^{2(1-\theta)}})>\frac{\alpha^2G^2c^2}{\epsilon^{2(1-\theta)}}$, then the same order of iteration complexity remains.  {If one aims to find a point $\w$ such that $\|\w - \w^*\|_2\leq\epsilon$, we can  apply RSG to find a solution $\w$ such that $f(\w) - f_*\leq (\epsilon/c)^{1/\theta}\leq \epsilon$ (where the last inequality is due to $\theta\leq 1$ and assuming $c\geq 1$ without loss of generality). Then under the local error bound condition, we have $\|\w - \w^*\|_2\leq c(f(\w) - f_*)^{\theta}\leq \epsilon$. For finding a solution $\w$ such that $f(\w) - f_*\leq (\epsilon/c)^{1/\theta}\leq \epsilon$, RSG requires an iteration complexity of $\widetilde O(\frac{1}{\epsilon^{2(1-\theta)/\theta}})$. Therefore, in order to find a solution $\w$ such that $\|\w - \w^*\|_2\leq\epsilon$, the iteration complexity of RSG is $\widetilde O(\frac{1}{\epsilon^{2(1-\theta)/\theta}})$. }

Next, we will consider different convex optimization problems that admit a local error bound on $\S_\epsilon$ with different  $\theta$ and show the faster convergence of RSG when applied to these problems.


\subsection{Linear Convergence for Polyhedral Convex Optimization}\label{sec:poly}
In this subsection, we consider a special family of non-smooth and non-strongly convex problems where the epigraph of $f(\cdot)$ over $\Omega$ is a polyhedron. In this case, we call \eqref{eqn:prob0} a \textbf{polyhedral convex minimization}
problem. We show that, in polyhedral convex minimization problem, $f(\cdot)$ has a linear growth property and admits a local error bound with $\theta=1$ so that $B_\epsilon\leq c\epsilon$ for a constant $c<\infty$.

\begin{lemma}[Polyhedral Error Bound Condition]\label{lem:polyeb}
Suppose $\Omega$ is a polyhedron and the epigraph of $f(\cdot)$ is also polyhedron. There exists a constant $\kappa>0$  such that
\[
\|\w  - \w^*\|_2\leq \frac{f(\w)-f_*}{\kappa}, \quad\forall \w\in\Omega.
\]
Thus, $f(\cdot)$ admits a local error bound on $\S_\epsilon$ with $\theta=1$ and $c=\frac{1}{\kappa}$~\footnote{In fact, this property of $f(\cdot)$ is a global error bound on $\Omega$.} (so $B_\epsilon \leq \frac{\epsilon}{\kappa}$) for any $\epsilon>0$.
\end{lemma}
{\bf Remark:} {The above inequality is also known as weak sharp minimum condition in literature~\citep{doi:10.1137/0331063,doi:10.1137/S0363012996301269,Ferris1991,DBLP:journals/mp/BurkeD02,DBLP:journals/mp/BurkeD05,DBLP:journals/mp/BurkeD09}. A proof of Lemma~\ref{lem:polyeb} is given by~\cite{doi:10.1137/0331063}. We also provide a proof in~\citep{DBLP:journals/corr/arXiv:1512.03107}.} We remark that the above result can be extended to any valid norm to measure the distance between $\w$ and $\w_*$.  Lemma~\ref{lem:polyeb}  generalizes Lemma 4 in \citep{DBLP:journals/mp/GilpinPS12}, which requires $\Omega$ to be a bounded polyhedron, to a similar result where $\Omega$ can be an unbounded polyhedron. This generalization is simple but useful because it helps the development of efficient algorithms based on this error bound for unconstrained problems without artificially including a box constraint.

Lemma~\ref{lem:polyeb} provides the basis for RSG to achieve a linear convergence for the polyhedral convex minimization problems. In fact, the following linear convergence of RSG can be obtained if we plugin the values of $\theta=1$ and $c=\frac{1}{\kappa}$ into Corollary~\ref{lem:leb_corr}.
\begin{corollary}\label{cor:5}
Suppose Assumption 1 holds and \eqref{eqn:prob0} is a polyhedral convex minimization
problem. The iteration complexity of RSG for obtaining an $\epsilon$-optimal solution is $O(\frac{\alpha^2G^2}{\kappa^2}\lceil \log_\alpha(\frac{\epsilon_0}{\epsilon})\rceil)$ provided $t =  \frac{\alpha^2 G^2}{\kappa^2}$ and $K = \lceil \log_\alpha(\frac{\epsilon_0}{\epsilon})\rceil$.
\end{corollary}
We want to point out that Corollary~\ref{cor:5} can be proved directly by replacing $\w^\dagger_{k-1,\epsilon}$ by $\w^*_{k-1}$ and replacing $\rho_\epsilon$ by $\kappa$ in the proof of Theorem~\ref{thm:GDr}. Here, we derive it as a corollary of a more general result. We also want to mention that, as shown by \cite{2015arXivJames}, the linear convergence rate in Corollary~\ref{cor:5} can  be also obtained by a SG method for the historically best solution, provided  $f_*$ is known.

\paragraph{\bf Examples.} Many non-smooth and non-strongly convex machine learning problems satisfy the assumptions of Corollary~\ref{cor:5}, for example, {\bf $\ell_1$ or $\ell_\infty$ constrained or regularized piecewise linear loss minimization}. In many machine learning tasks (e.g., classification and regression), there exists a set of data $\{(\x_i,y_i)\}_{i=1,2,\dots,n}$ and one often needs to solve the following empirical risk minimization problem
\begin{eqnarray*}
\min_{\w\in\R^d}f(\w) \triangleq \frac{1}{n}\sum_{i=1}^n\ell(\w^{\top}\x_i, y_i) + R(\w),
\end{eqnarray*}
where $R(\w)$ is a regularization term and $\ell(z,y)$ denotes a loss function. We consider a special case where (a) $R(\w)$ is a $\ell_1$ regularizer, $\ell_\infty$ regularizer or an indicator function of a $\ell_1$/$\ell_\infty$ ball centered at zero; and (b) $\ell(z,y)$ is any piecewise linear loss function, including hinge loss $\ell(z,y)= \max(0, 1-yz)$, absolute loss $\ell(z,y) = |z- y|$, $\epsilon$-insensitive loss $\ell(z,y) = \max(|z-y| - \epsilon, 0)$, and etc~\citep{DBLP:journals/ml/YangMJZ14Non}. It is easy to show that the epigraph of $f(\w)$ is a polyhedron if $f(\w)$ is defined as a sum of  any of these regularization terms and any of these loss functions. In fact, a piecewise linear loss functions can be generally written as
\begin{equation}
\label{eq:polyell}
\ell(\w^{\top}\x, y) = \max_{1\leq j\leq m}a_{j}\w^{\top}\x+b_j,
\end{equation}
where $(a_j,b_j)$ for $j=1,2,\dots,m$ are finitely many pairs of scalars. The formulation \eqref{eq:polyell} indicates that $\ell(\w^{\top}\x, y)$ is a piecewise affine function so that its epigraph  is a polyhedron. In addition, the $\ell_1$ or $\ell_\infty$ norm is also a polyhedral function because we can represent them as
\[
\|\w\|_1 = \sum_{i=1}^d\max(w_i, -w_i), \quad \|\w\|_\infty=\max_{1\leq i\leq d}|w_i|  = \max_{1\leq i\leq d}\max(w_i, -w_i).
\]
Since the sum of finitely many polyhedral functions is also a polyhedral function, the epigraph of $f(\w)$ is a polyhedron. 

Another important family of problems whose objective function has a polyhedral epigraph is {\bf submodular function minimization}.  Let $V = \{1,\ldots, d\}$ be a set and $2^V$ denote its power set. A submodular function $F(A): 2^{V}\rightarrow\R$ is a set function such that $F(A) + F(B)\geq F(A\cup B)+ F(A\cap B)$ for all subsets $A, B\subseteq V$ and $F(\emptyset)=0$. A submodular function minimization can be cast into a non-smooth convex optimization using the Lov\'{a}sz extension~\citep{DBLP:journals/ftml/Bach13}. In particular,  let the base polyhedron $B(F)$  be defined as
\[
B(F) = \{\s\in\R^d, \s(V) = F(V), \forall A\subseteq V, \s(A)\leq F(A)\},
\]
where $\s(A) = \sum_{i\in A}s_i$.  Then the Lov\'{a}sz extension of $F(A)$ is  $f(\w) = \max_{\s\in B(F)}\w^{\top}\s$, and $\min_{A\subseteq V}F(A) = \min_{\w\in[0,1]^d}f(\w)$. As a result, a submodular function minimization is essentially a non-smooth and non-strongly convex optimization with a polyhedral epigraph. 

\subsection{Improved Convergence for Locally Semi-Strongly Convex Problems}
First, we give a definition of local semi-strong convexity. 
\begin{definition}\label{def:ssc}
A function $f(\w)$ is semi-strongly convex on the $\epsilon$-sublevel set $\S_\epsilon$ if there exists $\lambda>0$ such that
\begin{equation}\label{eqn:strong}
\frac{\lambda}{2}\|\w - \w^* \|_2^2\leq f(\w) - f(\w^*),\quad\forall \w\in\S_\epsilon
\end{equation}
where $\w^*$ is the closest point to $\w$ in the optimal set.
\end{definition}
We refer to the property~(\ref{eqn:strong}) as \emph{local} semi-strong convexity when $\S_\epsilon\neq\Omega$. The two papers \citep{DBLP:journals/corr/GongY14,DBLP:journals/corr/nesterov16linearnon} have explored the semi-strong convexity on the whole domain $\Omega$ to prove linear convergence of smooth optimization problems. In~\citep{DBLP:journals/corr/nesterov16linearnon}, the inequality~(\ref{eqn:strong}) is also called \textbf{second-order growth property}. They have also shown that a class of problems satisfy~(\ref{eqn:strong}) (see examples given below).
The inequality \eqref{eqn:strong} indicates that $f(\cdot)$ admits a local error bound on $\S_\epsilon$ with $\theta=\frac{1}{2}$ and $c=\sqrt{\frac{2}{\lambda}}$, which leads to the following the corollary about the iteration complexity of RSG for locally semi-strongly convex problems. 
\begin{corollary} Suppose Assumption~\ref{ass:rsg} holds and $f(\w)$ is semi-strongly convex on  $\S_\epsilon$. Then $B_\epsilon \leq \sqrt{\frac{2\epsilon}{\lambda}}$ \footnote{Recall \eqref{eqn:Bebound}.}
and the iteration complexity of RSG for obtaining an $2\epsilon$-optimal solution is  $O(\frac{2\alpha^2G^2}{\lambda \epsilon}\lceil \log_\alpha(\frac{\epsilon_0}{\epsilon})\rceil)$ provided $t =  \frac{2\alpha^2 G^2}{\lambda\epsilon}$ and $K = \lceil \log_\alpha(\frac{\epsilon_0}{\epsilon})\rceil$.
\end{corollary}
{\bf Remark:} 
Here, we obtain an $\widetilde O(1/\epsilon)$ iteration complexity ($\widetilde O(\cdot)$ suppresses constants and logarithmic terms) only with local semi-strong convexity. It is obvious that strong convexity implies local semi-strong convexity~\citep{hazan-20110-beyond} but not vice versa. 


\paragraph{\bf Examples} Consider a family of functions in the form of $f(\w)  = h(X\w) + r(\w)$, where $X\in\mathbb{R}^{n\times d}$, $h(\cdot)$ is \emph{strongly convex} on any \emph{compact set} and  $r(\cdot)$ has a polyhedral epigraph. According to~\citep{DBLP:journals/corr/GongY14,DBLP:journals/corr/nesterov16linearnon}, such a function $f(\w)$ satisfies~(\ref{eqn:strong}) for any $\epsilon\leq \epsilon_0$ with a constant value for $\lambda$. Although smoothness is assumed for $h(\cdot)$ in~\citep{DBLP:journals/corr/GongY14,DBLP:journals/corr/nesterov16linearnon}, we find that it is not necessary for proving~(\ref{eqn:strong}). We state this result as the lemma below.

\begin{lemma} Suppose Assumption~\ref{ass:rsg} holds, $\Omega=\{\w\in \mathbb{R}^{d}|C\w\leq \mathbf{b}\}$ with $C\in\mathbb{R}^{k\times d}$ and $\mathbf{b}\in\mathbb{R}^k$, and $f(\w)  = h(X\w) + r(\w)$ where $h:\mathbb{R}^n\rightarrow\mathbb{R}$ satisfies $\text{dom}(h)=\mathbb{R}^k$ and is a strongly convex function on any compact set in $\mathbb{R}^n$, and $r(\w)$  has a polyhedral epigraph. Then, $f(w)$ satisfies (\ref{eqn:strong}) for any
$\epsilon\leq \epsilon_0$.
\end{lemma}
The proof of this lemma can be found in~\citep{DBLP:journals/corr/GongY14,DBLP:journals/corr/nesterov16linearnon,doi:10.1137/130950288}. For example, it is almost identical to the proof of Lemma 1 in \citep{DBLP:journals/corr/GongY14} which assumes $h(\cdot)$ is smooth). However, a similar result holds without the smoothness of $h(\cdot)$.

The function of this type covers some commonly used loss functions and regularization terms in machine learning and statistics. For example, we can consider 
{\bf robust regression with/without $\ell_1$ regularizer}~\citep{DBLP:journals/tit/XuCM10,Bertsimas14characterizationof}:
\begin{equation}\label{eqn:rr}
\min_{\w\in\Omega}\frac{1}{n}\sum\limits_{i=1}^n|\x_i^{\top}\w - y_i|^p + \lambda\|\w\|_1,
\end{equation}
where $p\in(1,2)$,  $\x_i\in\R^d$ denotes the feature vector and $y_i$ is the target output. The objective function is in the form of $h(X\w) + r(\w)$ where $X$ is a $n\times d$ matrix with $\x_1,\x_2,\dots,\x_n$ being its rows and $h(\u):=\sum_{i=1}^n|u_i-y_i|^p$. According to~\citep{Goebel_localstrong}, $h(\u)$ is a  strongly convex function on any compact set so that the objective function above is semi-strongly convex on $\S_\epsilon$ for any $\epsilon\leq \epsilon_0$.



\subsection{Improved Convergence for  Convex Problems with KL property}

Lastly, we consider a family of non-smooth functions with a local Kurdyka-\L ojasiewicz (KL) property. The definition of KL property is given below.
 \begin{definition}\label{def:KL}
 The function $f(\w)$ has the Kurdyka - \L ojasiewicz (KL) property at $\bar\w$ if there exist $\eta\in(0,\infty]$, a neighborhood $U_{\bar\w}$ of $\bar\w$ and a continuous concave function $\varphi:[0, \eta) \rightarrow \R_+$ such that  (i) $\varphi(0) = 0$; (ii) $\varphi$ is continuous on $(0, \eta)$; (iii) for all $s\in(0, \eta)$, $\varphi'(s)>0$; (iv) and for all $\w\in U_{\bar\w}\cap \{\w: f(\bar\w)< f(\w)< f(\bar\w)+\eta\}$, the Kurdyka - \L ojasiewicz (KL) inequality holds
\begin{eqnarray}
\label{eq:KLineq}
 \varphi'(f(\w) - f(\bar\w))\|\partial f(\w)\|_2\geq 1,
 \end{eqnarray}
 where $\|\partial f(\w)\|_2:=\min_{\mathbf g\in\partial f(\w)}  \|\mathbf g\|_2$.
 \end{definition}
 The function $\varphi$ is called the \textbf{desingularizing function} of $f$ at $\bar\w$, which sharpens  the function $f(\w)$  by reparameterization. An important desingularizing function is in the form of $\varphi(s) = cs^{1-\beta}$ for some $c>0$ and $\beta\in[0,1)$, by which, \eqref{eq:KLineq} gives the KL inequality
 \[
\|\partial f(\w)\|_2\geq  \frac{1}{c(1-\beta)}(f(\w) - f(\bar\w))^{\beta}.
 \]
Note that all semi-algebraic functions satisfy the KL property at any point~\citep{Bolte:2014:PAL:2650160.2650169}. Indeed, all the concrete examples given before satisfy  the  Kurdyka - \L ojasiewicz property. For more discussions about the KL property, we refer readers  to~\citep{Bolte:2014:PAL:2650160.2650169,journals/siamjo/BolteDLS07,DBLP:journals/corr/SchneiderU14,journals/mp/AttouchBS13,Bolte:2006:LIN:1328019.1328299}.
The following corollary states the iteration complexity of RSG for unconstrained problems that have the KL property at each $\bar\w\in\Omega_*$ .
  \begin{corollary}
  \label{thm:KL}
  Suppose Assumption~\ref{ass:rsg}  holds,  $f(\w)$ satisfies a (uniform) Kurdyka  - \L ojasiewicz property at any $\bar\w\in\Omega_*$  with the same desingularizing function $\varphi$ and constant   $\eta$, and
  \begin{equation}\label{eqn:smallS}
  \S_\epsilon\subset\cup_{\bar\w\in\Omega_*}\left[U_{\bar\w}\cap \{\w: f(\bar\w)< f(\w)< f(\bar\w)+\eta\}\right].
  \end{equation}
 RSG has an iteration complexity of  $O\left(\alpha^2G^2(\frac{\varphi(\epsilon)}{\epsilon})^2\lceil \log_\alpha(\frac{\epsilon_0}{\epsilon})\rceil\right)$  for obtaining an $2\epsilon$-optimal solution 
 provided
 $t=\alpha^2G^2(\varphi(\epsilon)/\epsilon)^2$. 
In addition, if $\varphi(s) = cs^{1-\beta}$ for some $c>0$ and $\beta\in[0,1)$, the iteration complexity of RSG is
  $O( \frac{\alpha^2 G^2c^2(1-\beta)^2}{\epsilon^{2\beta}}\lceil \log_\alpha(\frac{\epsilon_0}{\epsilon})\rceil)$ provided  $t = \frac{\alpha^2G^2c^2}{\epsilon^{2\beta}}$ and $K = \lceil \log_\alpha(\frac{\epsilon_0}{\epsilon})\rceil$.
\end{corollary}
\begin{proof}
We can prove the above corollary following a result in~\citep{arxiv:1510.08234} as presented in Proposition~\ref{prop:KL} in the appendix. According to Proposition~\ref{prop:KL}, if $f(\cdot)$ satisfies the KL property at $\bar\w$, then for all $\w\in U_{\bar \w}\cap \{\w: f(\bar\w)< f(\w)< f(\bar \w)+ \eta\}$ it holds that $
\|\w - \w^*\|_2\leq \varphi(f(\w) - f(\bar \w))$.
It then, under the uniform condition in~(\ref{eqn:smallS}), implies that, for any $\w\in\S_\epsilon$
$$
 \|\w - \w^*\|_2 \leq \varphi(f(\w) - f_*)\leq \varphi(\epsilon),
$$
where we use the monotonic property of $\varphi$. Then the first conclusion follows similarly as Corollary~\ref{lem:3} by noting $B_\epsilon \leq  \varphi(\epsilon)$.
%
The second conclusion immediately follows by setting $\varphi(s) = cs^{1-\beta}$ in the first conclusion.  {Please note that the above inequality implies the local error bound condition with $\theta=1-\beta$ for $\varphi(s) = cs^{1-\beta}$.} 
\end{proof}

While the conclusion in Corollary~\ref{thm:KL} hinges on a condition in~(\ref{eqn:smallS}), in practice  many convex functions (e.g., continuous semi-algebraic or subanalytic functions) satisfy the KL property with $U =\R^d$ and any finite $\eta<\infty$~\citep{Attouch:2010:PAM:1836121.1836131,arxiv:1510.08234,DBLP:journals/siamjo/Li10}.

It is worth mentioning that to our best  knowledge, the present work is  the first to leverage the KL property for developing improved subgradient methods, though it has been explored in non-convex and convex optimization for deterministic  descent methods for smooth optimization~\citep{arxiv:1510.08234,Bolte:2014:PAL:2650160.2650169,Attouch:2010:PAM:1836121.1836131,DBLP:conf/pkdd/KarimiNS16}. {For example, \cite{arxiv:1510.08234} studied the convergence of {\bf subgradient descent sequence} for minimizing a convex function under an error bound condition.  A sequence $\{\x_k\}$ is called a subgradient descent sequence if there exist $a>0, b>0$ it satisfies two conditions,  namely  sufficient decrease condition $f(\x_k) + a \|\x_k - \x_{k-1}\|_2^2\leq f(\x_{k-1})$, and relative error condition, i.e.,  there exists $\omega_k\in\partial f(\x_k)$ such that  $
\|\omega_k\|_2 \leq b\|\x_k - \x_{k-1}\|_2$. 
However, for a general non-smooth  function $f(\x)$, the sequence generated by subgradient method, i.e., $\x_k = \x_{k-1} - \eta_k \partial\nabla f(\x_{k-1})$ do not necessarily satisfy the above two conditions. Instead, \cite{arxiv:1510.08234} considered proximal gradient method that only applies to a smaller family of functions consisting of a smooth component and a non-smooth component by assuming the proximal mapping for the non-smooth component can be efficiently computed. In contrast, our algorithm and analysis are developed for much general non-smooth functions.}

\section{Variants of RSG without knowing the constant $c$ and the exponent $\theta$ in the local error bound}\label{sec:prac}
In Section~\ref{sec:spec}, we have discussed the local error bound and presented several classes of problems to reveal the magnitude of $B_\epsilon$, i.e., $B_\epsilon = c\epsilon^\theta$. For some problems, the value of $\theta$ is exhibited. However, the value of the constant $c$ could be still difficult to estimate, which renders it challenging to set the appropriate value $t= \frac{\alpha^2c^2 G^2}{\epsilon^{2(1-\theta)}}$ for inner iterations of RSG. In practice, one might use a sufficiently large $c$ to set up the value of $t$. However, such an approach might be  vulnerable to both over-estimation and under-estimation of $t$. Over-estimating the value of $t$ leads to a waste of iterations while under-estimation leads to an less accurate solution that might not reach to the target accuracy level.  In addition, for some problems the value of $\theta$ is still an open problem. One interesting family of objective functions in machine learning is the sum of piecewise linear loss over training data and a nuclear norm regularizer or an overlapped or non-overlapped group lasso regularizer.  In this section, we present variants of RSG that can be implemented without knowing the value of $c$ in the local error bound condition and even the value of exponent  $\theta$, and prove their improved convergence over the SG method.

\subsection{RSG without knowing $c$}The key idea is to use an increasing sequence of $t$ and another level of restarting for RSG. The detailed steps are presented in Algorithm~\ref{alg:2}, to which we refer as R$^2$SG. With large enough $t_1$ in R$^2$SG, the complexity of R$^2$SG for finding an $\epsilon$ solution is given by the theorem below.

\begin{theorem}\label{thm:R2SG}
Suppose $\epsilon\leq \epsilon_0/4$ and $K =\lceil \log_\alpha(\epsilon_0/\epsilon)\rceil$. Let $t_1$ in Algorithm~\ref{alg:2} be large enough so that there exists $\hat\epsilon_1\in(\epsilon, \epsilon_0/2)$,  with which $f(\cdot)$ satisfies a local error bound condition on $\mathcal S_{\hat\epsilon_1}$ with $\theta \in (0,1)$ and the constant $\hat c$, and $t_1 = \frac{\alpha^2\hat c^2G^2}{\hat\epsilon_1^{2(1-\theta)}}$. Then, with at most $S=\lceil\log_2(\hat\epsilon_1/\epsilon)\rceil +1$ calls of RSG in Algorithm~\ref{alg:2}, we find a solution $\w^S$ such that $f(\w^S) - f_*\leq 2\epsilon$. The total number of iterations of R$^2$SG for obtaining $2\epsilon$-optimal solution is upper bounded by $T_S = O\left(\frac{\hat c^2G^2}{\epsilon^{2(1-\theta)}}\lceil \log_\alpha(\frac{\epsilon_0}{\epsilon})\rceil\right)$.
\end{theorem}
\begin{proof}
Since $K=\lceil\log_\alpha(\epsilon_0/\epsilon)\rceil\geq \lceil\log_\alpha(\epsilon_0/\hat\epsilon_1)\rceil$ and $t_1 = \frac{\alpha^2\hat c^2G^2}{\hat\epsilon_1^{2(1-\theta)}}$, we can apply Corollary~\ref{lem:leb_corr} with $\epsilon=\hat\epsilon_1$ to the first call of RSG in Algorithm~\ref{alg:2} so that the output $\w^1$ satisfies
\begin{equation}\label{eqn:RSG-s1}
f(\w^1) - f_* \leq 2\hat\epsilon_1.
\end{equation}
Then, we consider the second call of RSG with the initial solution $\w^1$ satisfying~(\ref{eqn:RSG-s1}). By the setup  $K = \lceil\log_\alpha(\epsilon_0/\epsilon)\rceil\geq \lceil\log_\alpha(2\hat\epsilon_1/(\hat\epsilon_1/2))\rceil$ and $t_2 =t_1 2^{2(1-\theta)}= \frac{\hat c^2G^2}{(\hat\epsilon_1/2)^{2(1-\theta)}}$, we can apply Corollary~\ref{lem:leb_corr} with $\epsilon=\hat\epsilon_1/2$ and $\epsilon_0=2\hat\epsilon_1$ so that the output $\w^2$ of the second call satisfies $
f(\w^2) - f_* \leq \hat\epsilon_1$.
By repeating this argument for all the subsequent calls of RSG,  with at most $S = \lceil \log_2(\hat\epsilon_1/\epsilon)\rceil + 1$ calls, Algorithm~\ref{alg:2} ensures that
\[
f(\w^S) - f_* \leq 2\hat\epsilon_1/2^{S-1} \leq 2\epsilon
\]
The total number of iterations during the $S$ calls of RSG is bounded by
\begin{equation*}
\begin{aligned}
T_S &= K\sum_{s=1}^St_s  =K \sum_{s=1}^St_12^{2(s-1)(1-\theta)} = Kt_1 2^{2(S-1)(1-\theta)}\sum_{s=1}^S \left(\frac{1}{2^{2(1-\theta)}}\right)^{S-s}\\
&\leq \frac{Kt_12^{2(S-1)(1-\theta)}}{1 - 1/2^{2(1-\theta)}} \leq O\left(Kt_1 \left(\frac{\hat\epsilon_1}{\epsilon}\right)^{2(1-\theta)}\right) = O\left(\frac{\hat c^2G^2}{\epsilon^{2(1-\theta)}}\lceil \log_\alpha(\frac{\epsilon_0}{\epsilon})\rceil\right).
\end{aligned}
\end{equation*}
\end{proof}

\begin{algorithm}[t]
\caption{RSG with restarting: R$^2$SG} \label{alg:2}
\begin{algorithmic}[1]
\STATE \textbf{Input}:  the  number of iterations $t_1$ in each stage of the first call of RSG and the number of stages $K$ in each call of RSG
\STATE \textbf{Initialization: } $\w^0 \in \Omega$;
\FOR{$s=1,2\ldots, S$}
\STATE Let $\w^s= \text{RSG}(\w^{s-1}, K, t_s, \alpha)$
\STATE Let $t_{s+1} = t_s 2^{2(1-\theta)}$
\ENDFOR
\end{algorithmic}
\end{algorithm}

\noindent\textbf{Remark:} We make several remarks about Algorithm~\ref{alg:2} and Theorem~\ref{thm:R2SG}: (i) Theorem~\ref{thm:R2SG} applies only when $\theta\in(0,1)$. If $\theta=1$, in order to have an increasing sequence of $t_s$, we can set $\theta$ in Algorithm~\ref{alg:2} to a little smaller value than $1$ in practical implementation, and the iteration complexity in Theorem~\ref{thm:R2SG} implies that R$^2$SG can enjoy a convergence rate close to linear convergence for problems satisfying the weak sharp minimum condition. (ii) the $\epsilon_0$ in the implementation of RSG (Algorthm~\ref{alg:1}) can be re-calibrated for $s\geq 2$ to improve the performance (e.g., one can use the relationship $f(\w_{s-1}) - f_* = f(\w_{s-2}) - f_* + f(\w_{s-1}) - f(\w_{s-2})$ to do re-calibration); (iii) as a tradeoff, the exiting criterion of R$^2$SG is not as automatic as RSG. In fact, the total number of calls $S$ of RSG for obtaining an $2\epsilon$-optimal solution depends on an unknown parameter (namely $\hat\epsilon_1$). In practice, one could use other stopping criteria to terminate the algorithm. For example, in machine learning applications one can monitor the performance on the validation data set to terminate the algorithm. (vi) The quantities $\hat\epsilon_1$, $S$  in the proof above are implicitly determined by $t_1$ and one does not need to compute $\hat\epsilon_1$ and $S$ in order to apply Algorithm~\ref{alg:2}. Finally, we note that when a local strong convexity condition holds on $\S_{\hat\epsilon_1}$ with $\hat\epsilon_1\geq \epsilon$ one might derive an iteration complexity of $O(1/\epsilon)$ for SG  by first showing that SG converges to $\S_{\hat\epsilon_1}$ with a number of iterations independent of $\epsilon$, then showing that  the iterates stay within  $\S_{\hat\epsilon_1}$ and converge to an $\epsilon$-level set with an iteration complexity of $O(1/\epsilon)$ following existing analysis of SG for strongly convex functions (e.g.,~\cite{DBLP:journals/corr/abs-1212-2002}). However, it still needs to know the value of the local strong convexity parameter unlike our result in Theorem~\ref{thm:R2SG} that does not need to known the local strong convexity parameter. 

\subsection{RSG for unknown $\theta$ and $c$}
Without knowing $\theta\in(0,1]$ and $c$ to get a sharper local error bound, we can simply  let $\theta=0$ and $c = B_{\epsilon'}$ with $\epsilon'\geq \epsilon$, which still render  the inequaity~(\ref{eqn:ler}) hold (c.f. Definition \ref{def:ler}).  Then we can employ the same trick to increase the values of $t$. In particular, we start with a sufficiently large value of $t$ and run RSG with $K=\lceil \log_\alpha(\epsilon_0/\epsilon)\rceil$ stages, and then increase the value of $t$  by a factor of $4$ and repeat the process. 

\begin{theorem}\label{thm:2RSG}
Let  $\theta=0$  in Algorithm~\ref{alg:2} and suppose $\epsilon\leq \epsilon_0/4$ and $K =\lceil \log_\alpha(\epsilon_0/\epsilon)\rceil$. Assume  $t_1$ in Algorithm~\ref{alg:2} is large enough so that there exists $\hat\epsilon_1\in(\epsilon, \epsilon_0/2]$ giving $t_1 = \frac{\alpha^2B_{\hat\epsilon_1}^2G^2}{\hat\epsilon_1^{2}}$.  Then, with at most $S=\lceil\log_2(\hat\epsilon_1/\epsilon)\rceil +1$ calls of RSG in Algorithm~\ref{alg:2}, we find a solution $\w^S$ such that $f(\w^S) - f_*\leq 2\epsilon$. The total number of iterations of R$^2$SG for obtaining $2\epsilon$-optimal solution is upper bounded by $T_S = O\left(\frac{B_{\hat\epsilon_1}^2G^2}{\epsilon^{2}}\lceil \log_\alpha(\frac{\epsilon_0}{\epsilon})\rceil\right)$.
\end{theorem}
{\bf Remark:} Since $B_\epsilon/\epsilon$ is a monotonically decreasing function in $\epsilon$~\citep[Lemma 7]{DBLP:journals/corr/abs-1607-01027}, such a $t_1$ in Theorem~\ref{thm:2RSG} exists. {Note that if the problem satisfies a KL property as in Corollary~\ref{thm:KL} and the value of $\beta$ is unknown, the above theorem still holds.}

\begin{proof}
The proof is similar to that of Theorem~\ref{thm:R2SG} except that we let $c = B_{\hat\epsilon_1}$ and $\theta=0$. Since $K=\lceil\log_\alpha(\epsilon_0/\epsilon)\rceil\geq \lceil\log_\alpha(\epsilon_0/\hat\epsilon_1)\rceil$ and $t_1 = \frac{\alpha^2B_{\hat\epsilon_1}^2G^2}{\hat\epsilon_1^{2}}$, we can apply Corollary~\ref{lem:3} with $\epsilon=\hat\epsilon_1$  to the first call of RSG in Algorithm~\ref{alg:2} so that the output $\w^1$ satisfies
\begin{equation}\label{eqn:RSG-s2}
f(\w^1) - f_* \leq 2\hat\epsilon_1.
\end{equation}
Then, we consider the second call of RSG with the initial solution $\w^1$ satisfying~(\ref{eqn:RSG-s2}). By the setup  $K = \lceil\log_\alpha(\epsilon_0/\epsilon)\rceil\geq \lceil\log_\alpha(2\hat\epsilon_1/(\hat\epsilon_1/2))\rceil$ and $t_2 =t_1 2^{2}= \frac{B_{\hat\epsilon_1}^2G^2}{(\hat\epsilon_1/2)^{2}}$, we can apply Corollary~\ref{lem:3} with $\epsilon=\hat\epsilon_1/2$ and $\epsilon_0=2\hat\epsilon_1$ (noting that $B_{\hat\epsilon_1}>B_{\hat\epsilon_1/2}$) so that the output $\w^2$ of the second call satisfies $
f(\w^2) - f_* \leq \hat\epsilon_1$.
By repeating this argument for all the subsequent calls of RSG,  with at most $S = \lceil \log_2(\hat\epsilon_1/\epsilon)\rceil + 1$ calls, Algorithm~\ref{alg:2} ensures that
\[
f(\w^S) - f_* \leq 2\hat\epsilon_1/2^{S-1} \leq 2\epsilon.
\]
The total number of iterations during the $S$ calls of RSG is bounded by
\begin{equation*}
\begin{aligned}
T_S &= K\sum_{s=1}^St_s  =K \sum_{s=1}^St_12^{2(s-1)} = Kt_1 2^{2(S-1)}\sum_{s=1}^S \left(\frac{1}{2^{2}}\right)^{S-s}\\
&\leq \frac{Kt_12^{2(S-1)}}{1 - 1/2^{2}} \leq O\left(Kt_1 \left(\frac{\hat\epsilon_1}{\epsilon}\right)^{2}\right) = O\left(\frac{B_{\hat\epsilon_1}^2G^2}{\epsilon^{2}}\lceil \log_\alpha(\frac{\epsilon_0}{\epsilon})\rceil\right).
\end{aligned}
\end{equation*}
\end{proof}

\section{Discussions and Comparisons}\label{sec:diss}
In this section, we further discuss the obtained results and compare them with existing results. 

\paragraph{Comparison with the standard SG} The standard SG's iteration complexity is known as $O(\frac{G^2\|\w_0 - \w^*_0\|_2^2}{\epsilon^2})$ for achieving an $2\epsilon$-optimal solution. By assuming $t$ is appropriately set  in RSG according to Corollary~\ref{lem:3},  its iteration complexity is $O(\frac{G^2B_{\epsilon}^2}{\epsilon^2}\log(\epsilon_0/\epsilon))$, which depends on $B_\epsilon^2$ instead of $\|\w_0 - \w_0^*\|_2^2$ and only has a logarithmic dependence on $\epsilon_0$, the upper bound of $f(\w_0) - f_*$. When the initial solution is far from the optimal set so that $B_\epsilon^2\ll\|\w_0 - \w_0^*\|_2^2$,  RSG could have a lower worst-case complexity. Even if $t$ is not appropriately set up to be larger than $\alpha^2G^2B_\epsilon^2/\epsilon^2$, Theorem~\ref{thm:2RSG} guarantees that the proposed R$^2$SG could still has a lower iteration complexity than that of SG as long as $t_1$ is sufficiently large. In some special cases, e.g., when $f$ satisfies the local error bound condition~(\ref{eqn:ler}) with $\theta\in(0,1]$, RSG only needs $O\left(\frac{1}{\epsilon^{2(1-\theta)}}\log\left(\frac{1}{\epsilon}\right)\right)$  iterations (see Corollary~\ref{lem:leb_corr} and Theorem~\ref{thm:R2SG}), which has a better dependency on $\epsilon$ than the complexity of standard SG method.

\paragraph{Comparison with the SG method in~\citep{2015arXivRobert}} \cite{2015arXivRobert} introduced a similar but different growth condition:
\begin{eqnarray}
\label{eqn:FreundCond}
\|\w-\w^*\|_2\leq \mathcal{G}\cdot(f(\w)-f_{slb}), \quad \forall \w\in\Omega,
\end{eqnarray}
where $f_{slb}$ is a strict lower bound of $f_*$, by \citet{2015arXivRobert}. The main differences from our key condition~(\ref{eqn:keyii}) are: the left-hand side is the distance of $\w$ to the optimal set in \eqref{eqn:FreundCond} while it is the distance of $\w$ to the $\epsilon$-sublevel set in \eqref{eqn:keyii}; the right-hand side is the objective gap with respect to $f_{slb}$ in \eqref{eqn:FreundCond} and it is the objective gap with respect to $f_*$ in \eqref{eqn:keyii}; the growth constant $\mathcal{G}$ in \eqref{eqn:FreundCond} varies with $f_{slb}$ and $\rho_\epsilon$ in \eqref{eqn:keyii} may depend on $\epsilon$ in general.

Freund and Lu's SG method has an iteration complexity of $O(G^2\mathcal{G}^2(\frac{\log H}{\epsilon'} + \frac{1}{\epsilon'^2}))$ for finding a solution $\hat{\w}$ such that $f(\hat{\w}) - f_*\leq \epsilon'(f_* - f_{slb})$, where $f_{slb}$ and $\mathcal{G}$ are defined in \eqref{eqn:FreundCond} and $H = \frac{f(\w_0)-f_{slb}}{f_*-f_{slb}}$. In comparison, our RSG can be better if $f_*-f_{slb}$ is large. To see this, we represent the complexity in~\citep{2015arXivRobert} in terms of the absolute error $\epsilon$ with $\epsilon=\epsilon'(f_* - f_{slb})$ and obtain $O(G^2\mathcal{G}^2(\frac{(f_*-f_{slb})\log H}{\epsilon} + \frac{(f_*-f_{slb})^2}{\epsilon^2}))$. If
the gap $f_* - f_{slb}$ is large, e.g., $O(f(\w_0)-f_{slb})$, the second term is dominating, which is at least $\Omega(\frac{G^2\|\w_0 - \w^*_0\|_2^2}{\epsilon^2})$ due to the definition of $\mathcal{G}$ in $\eqref{eqn:FreundCond}$. This complexity has the same order of magnitude as the standard SG method so that RSG can be better due to  the reasoning in last paragraph. {More generally, the iteration complexity of Freund and Lu's SG method can be reduced to $O(\frac{G^2B^2_{f_* - f{slb}}}{\epsilon^2})$ by choosing the best $\mathcal G$ in the proof of Theorem 1.1 in~\citep{2015arXivRobert}, which depends on $f_* - f_{slb}$. In comparison, RSG could have a lower complexity if $f_*-f_{slb}$ is larger than $\epsilon$ as in Corollary~\ref{lem:3} or $\hat\epsilon_1$ as in Theorem~\ref{thm:R2SG}. Our experiments in subsection~\ref{subsec:compFL} also corroborate this point.  In addition, RSG can leverage the local error bound condition to enjoy a lower iteration complexity than $O(1/\epsilon^2)$.}  

\paragraph{Comparison with \citep{Nesterov:2014:uniform_convex} for uniformly convex function} \cite{Nesterov:2014:uniform_convex} considered primal-dual subgradient methods for solving the problem  \eqref{eqn:prob0} with $f$ being \emph{uniformly convex}, namely,
$$
f(\alpha\w+ (1-\alpha)\v)\leq \alpha f(\w)+(1-\alpha)f(\v)-\frac{1}{2}\mu\alpha(1-\alpha)[\alpha^{\rho-1}+(1-\alpha)^{\rho-1}]\|\w-\v\|_2^\rho
$$
for any $\w$ and $\v$ in $\Omega$ and any $\alpha\in[0,1]$\footnote{The Euclidean norm in the definition here can be replaced by a general norm as in~\citep{Nesterov:2014:uniform_convex}.}, where $\rho\in[2,+\infty]$ and $\mu\geq0$. In this case, the method by~\citep{Nesterov:2014:uniform_convex} has an iteration complexity of $O\left(\frac{G^2}{\mu^{2/\rho}\epsilon^{2(\rho-1)/\rho}}\right)$. The uniform convexity of $f$ further implies $f(\w)-f_*\geq\frac{1}{2}\mu\|\w-\w^*\|_2^\rho$ for any $\w\in\Omega$ so that $f(\cdot)$ admits a local error bound on the $\epsilon$-sublevel set $\S_\epsilon$ with $c=\left(\frac{2}{\mu}\right)^{\frac{1}{\rho}}$ and $\theta=\frac{1}{\rho}$. Therefore, our RSG has a complexity of $O\left(\frac{G^2}{\mu^{2/\rho}\epsilon^{2(\rho-1)/\rho}}\log(\frac{\epsilon_0}{\epsilon})\right)$ according to Corollary~\ref{lem:leb_corr}. Compared to~\citep{Nesterov:2014:uniform_convex}, our complexity is higher by a logarithmic factor. However, we only require the local error bound property of $f$ that is weaker than uniform convexity and also covers much broader family of functions.  {Note that the above comparison is fair, since for achieving a target $\epsilon$-optimal solution the algorithms presented in \citep{Nesterov:2014:uniform_convex} do need the knowledge of  uniform convexity parameter $\rho$ and the parameter $\mu$.  It is worth mentioning that \cite{Nesterov:2014:uniform_convex} also presented algorithms with a fixed number of iterations $T$ as input  that achieve adaptive rates without knowledge of $\rho$ and $\mu$.  However, they only considered the case when $\rho>=2$, which correponds to $\theta\leq 1/2$ in our notations, while our methods can be applied also when $\theta> 1/2$.}

\section{Experiments}\label{sec:exp}
In this section, we present some  experiments to demonstrate the effectiveness of RSG. We first consider several applications in machine learning, in particular regression, classification and matrix completion, and focus on the comparison between RSG and SG. Then we make comparison between  RSG with  Freund \& Lu's SG variant for solving regression problems. In experiments, all compared algorithms use the same initial solution unless otherwise specified.

\begin{figure}[t]
\centering
\subfigure[different $t$]{\includegraphics[scale=0.28]{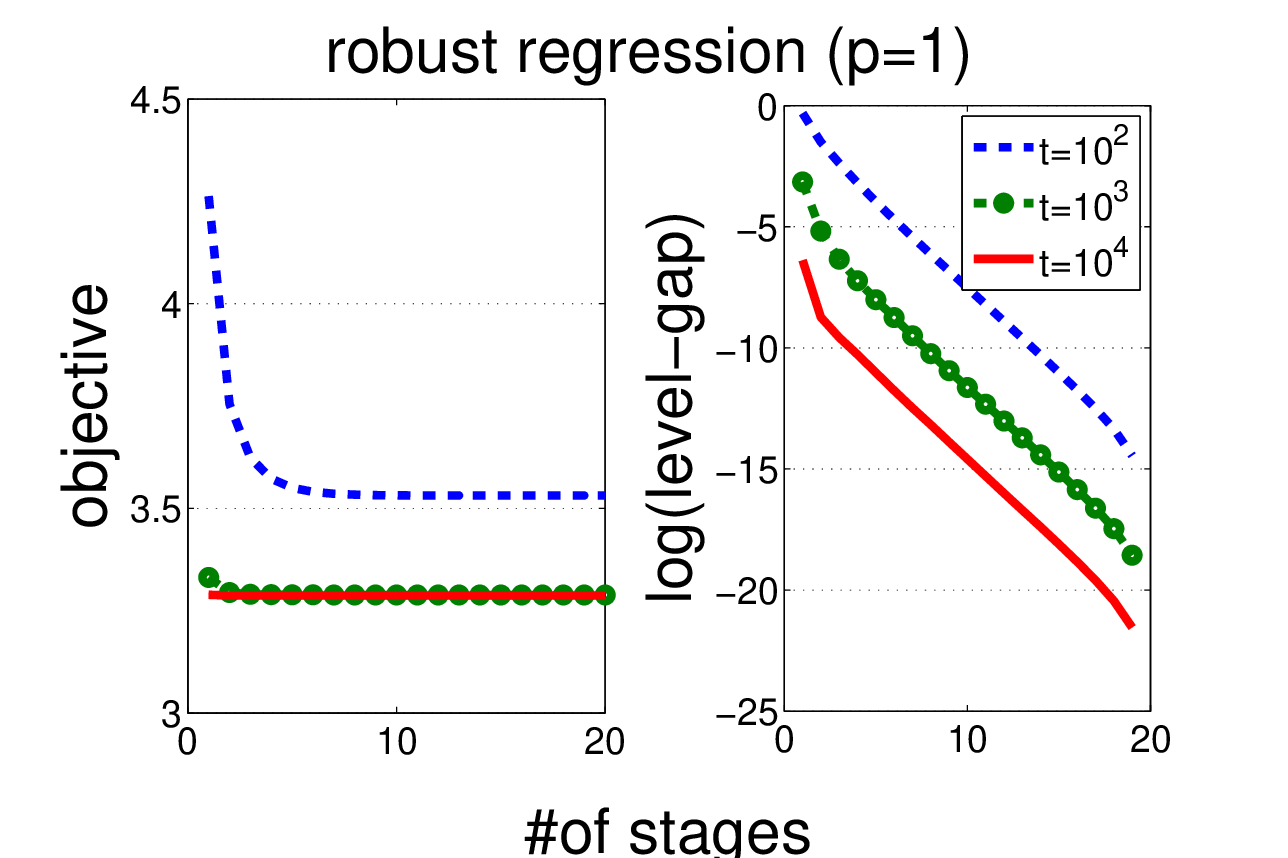}}\hspace*{0.1in}
\subfigure[different $t$]{\includegraphics[scale=0.28]{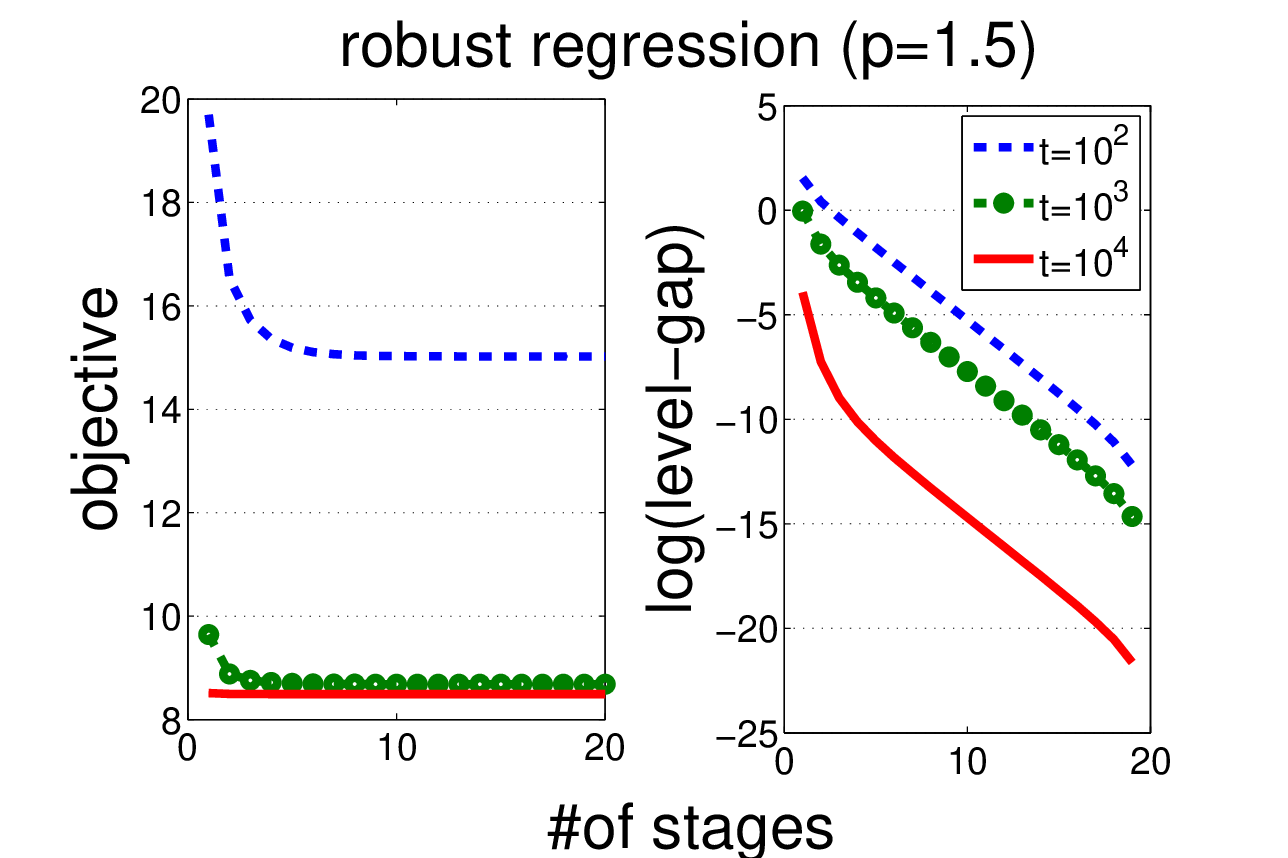}}\hspace*{0.1in}

\subfigure[different algorithms]{\includegraphics[scale=0.3]{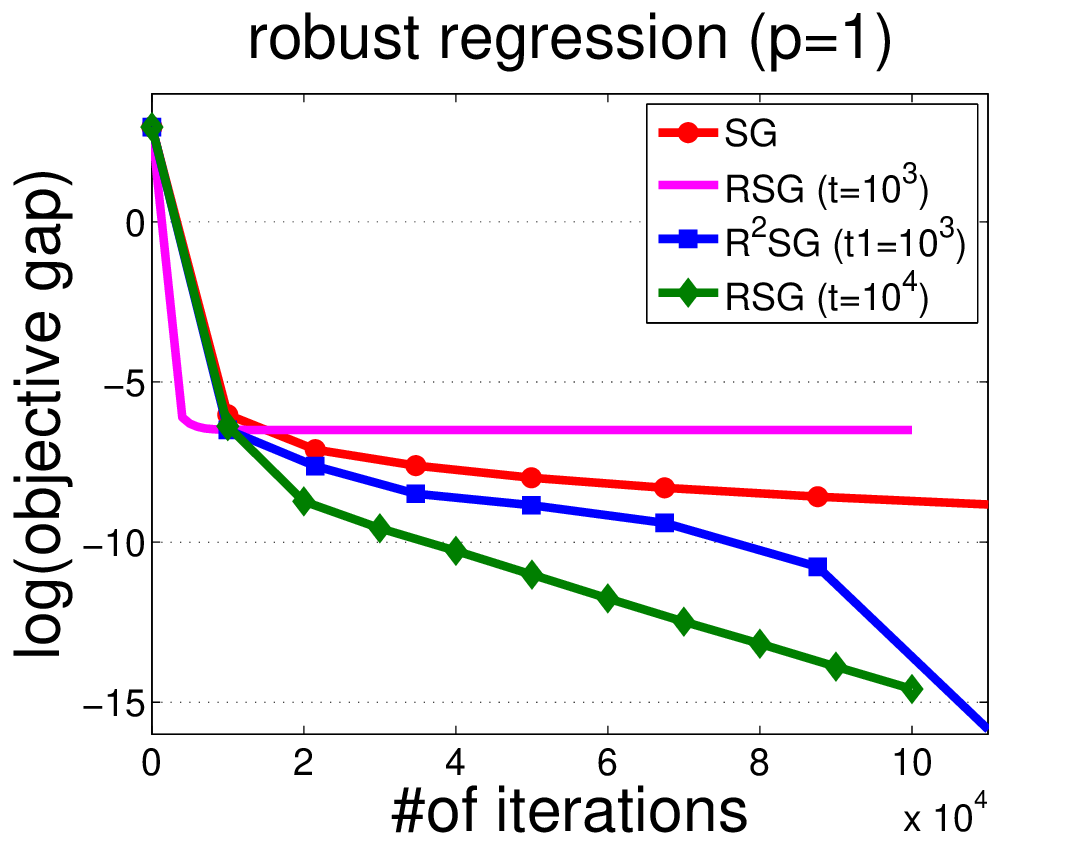}}\hspace*{0.1in}
\subfigure[different algorithms]{\includegraphics[scale=0.3]{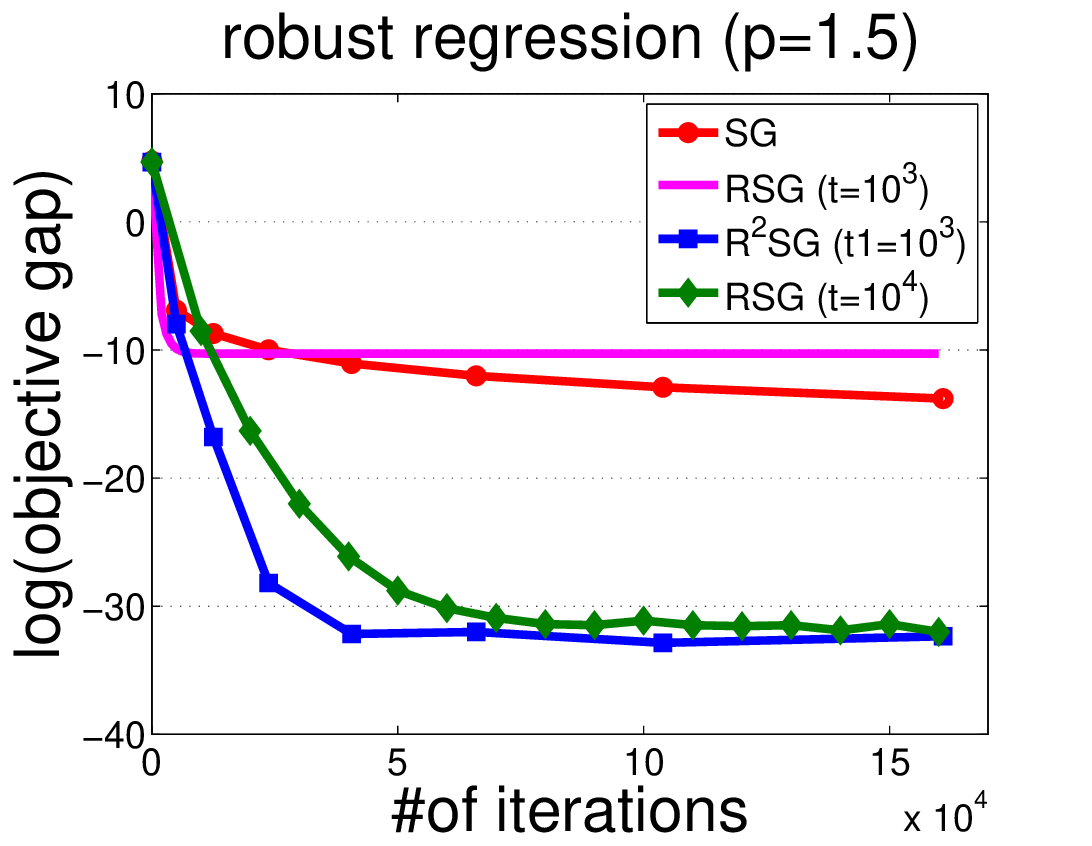}}\hspace*{0.1in}
\caption{Comparison of RSG with different $t$ and of different algorithms on the housing data. One iteration means one subgradient update in all algorithms. (t1 for R$^2$SG represents the initial value of $t$ in the first call of RSG.)  }\label{fig:4}
\vspace*{-0.22in}
\end{figure}

\subsection{Robust Regression}\label{subsec:RR} The regression problem is to predict an output $y$ based on a feature vector $\x\in\R^d$. Given a set of training examples $(\x_i,y_i), i=1,\ldots, n$, a linear  regression  model can be found  by solving the  optimization problem in~(\ref{eqn:rr}).

We solve two instances of the  problem with $p=1$ and $p=1.5$ and $\lambda=0$. We conduct experiments on two data sets from libsvm website~\footnote{\url{https://www.csie.ntu.edu.tw/~cjlin/libsvmtools/datasets/}}, namely  housing ($n=506$ and $d=13$) and  space-ga ($n=3107$ and $d=6$).    We first examine the convergence behavior of  RSG with different values for the number of iterations per-stage  $t=10^2$, $10^3$, and $10^4$. The value of $\alpha$ is set to $2$ in all experiments. The initial step size of RSG is set to be proportional to $\epsilon_0/2$ with the same scaling parameter for different variants.  We plot the results on housing data in Figure~\ref{fig:4} (a,b) and on space-ga data  in Figure~\ref{fig:3} (a,b). In each figure, we plot the objective value vs number of stages and the log difference between the objective value and the converged value (to which we refer as level gap). We can clearly see that with different values of  $t$ RSG converges to an $\epsilon$-level set and the convergence rate is linear in terms of the number of stages, which is consistent with our theory.

\begin{figure}[t]
\centering
\subfigure[different $t$]{\includegraphics[scale=0.28]{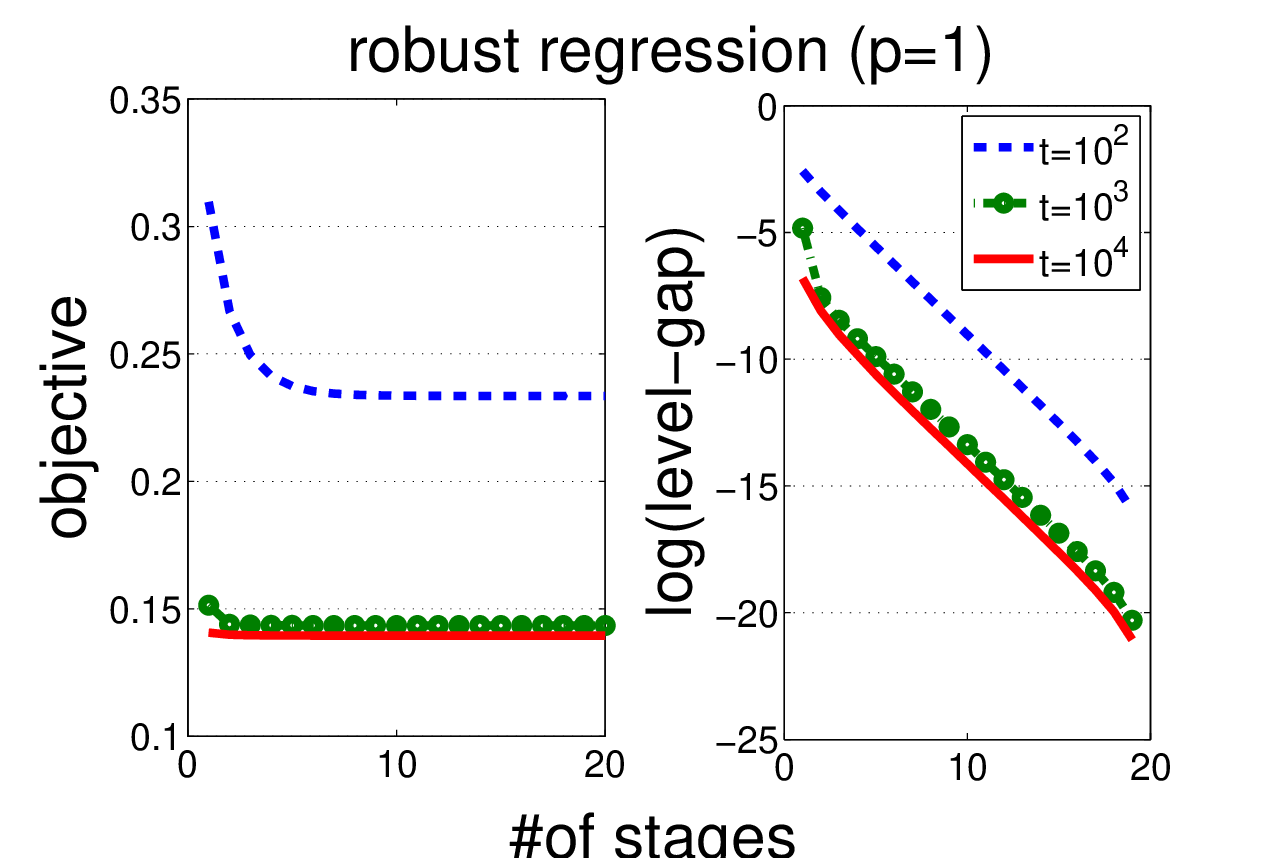}}\hspace*{0.1in}
\subfigure[different $t$]{\includegraphics[scale=0.28]{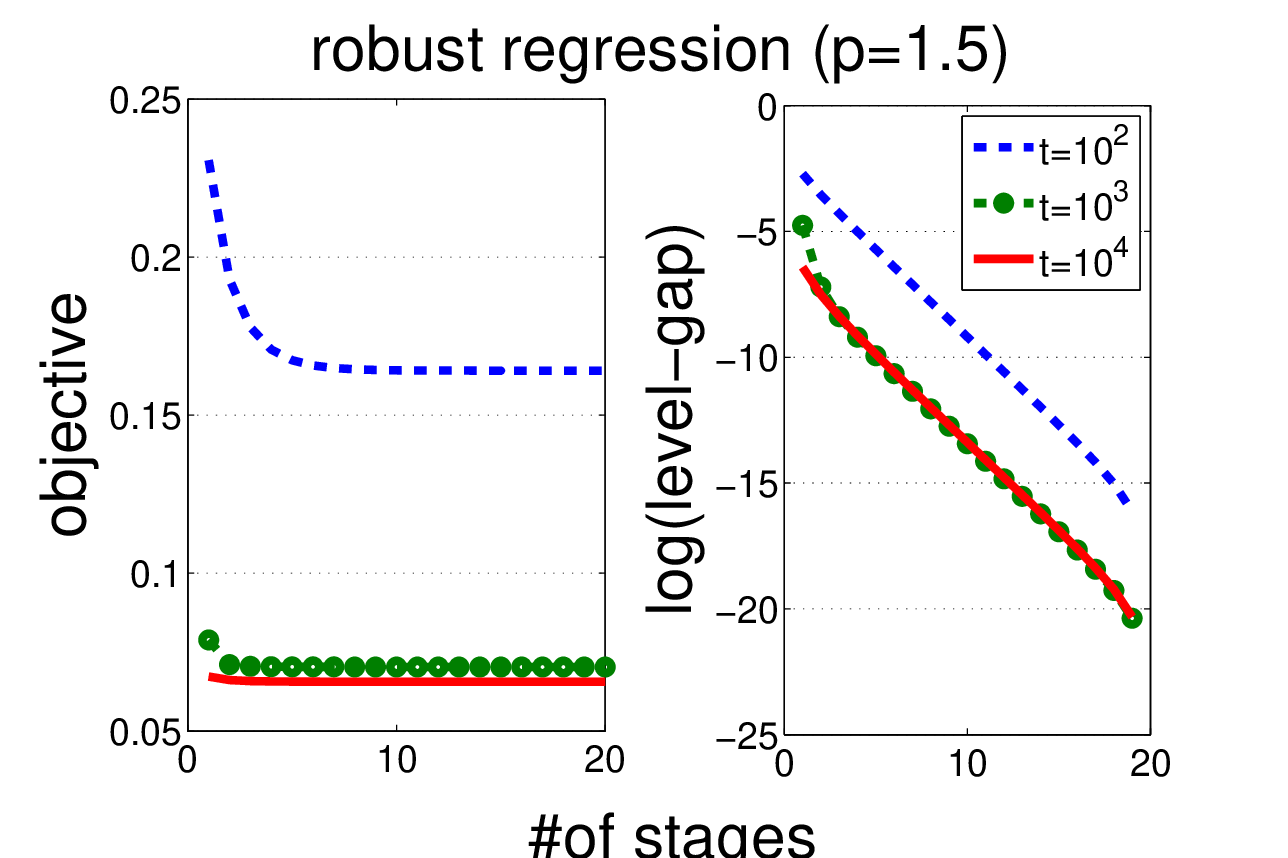}}\hspace*{0.1in}

\subfigure[different algorithms]{\includegraphics[scale=0.3]{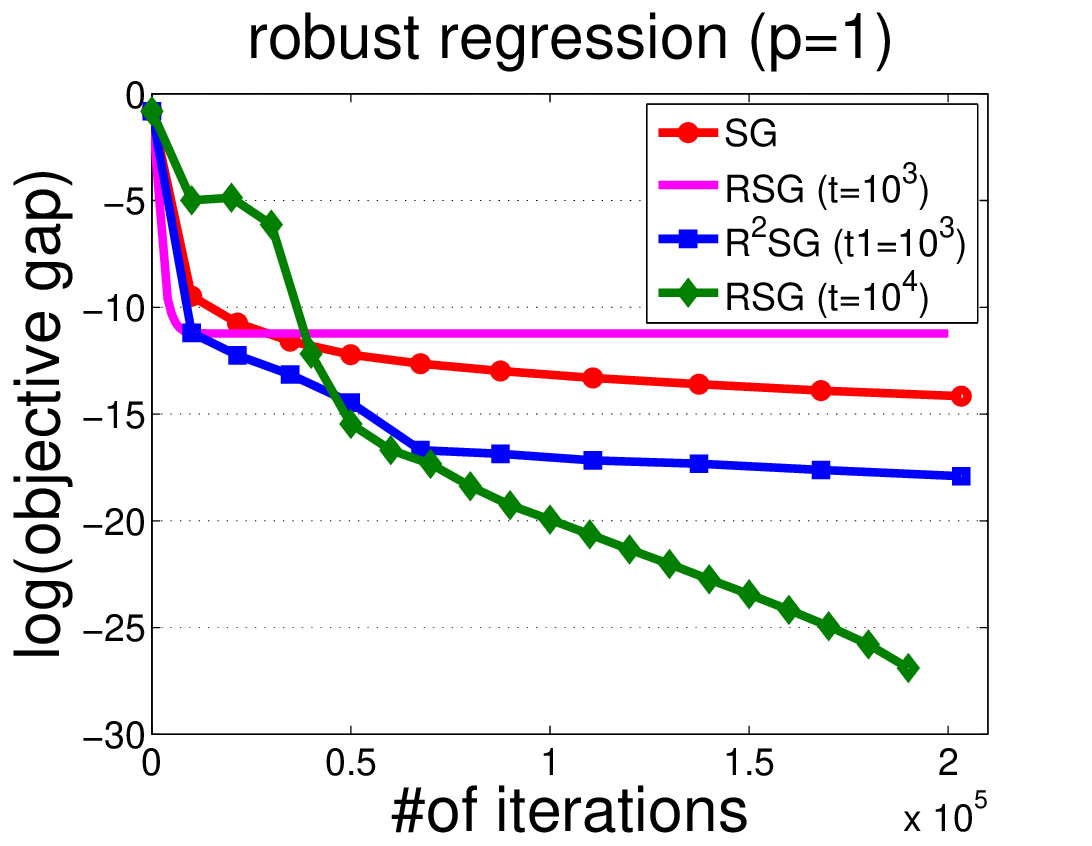}}\hspace*{0.1in}
\subfigure[different algorithms]{\includegraphics[scale=0.3]{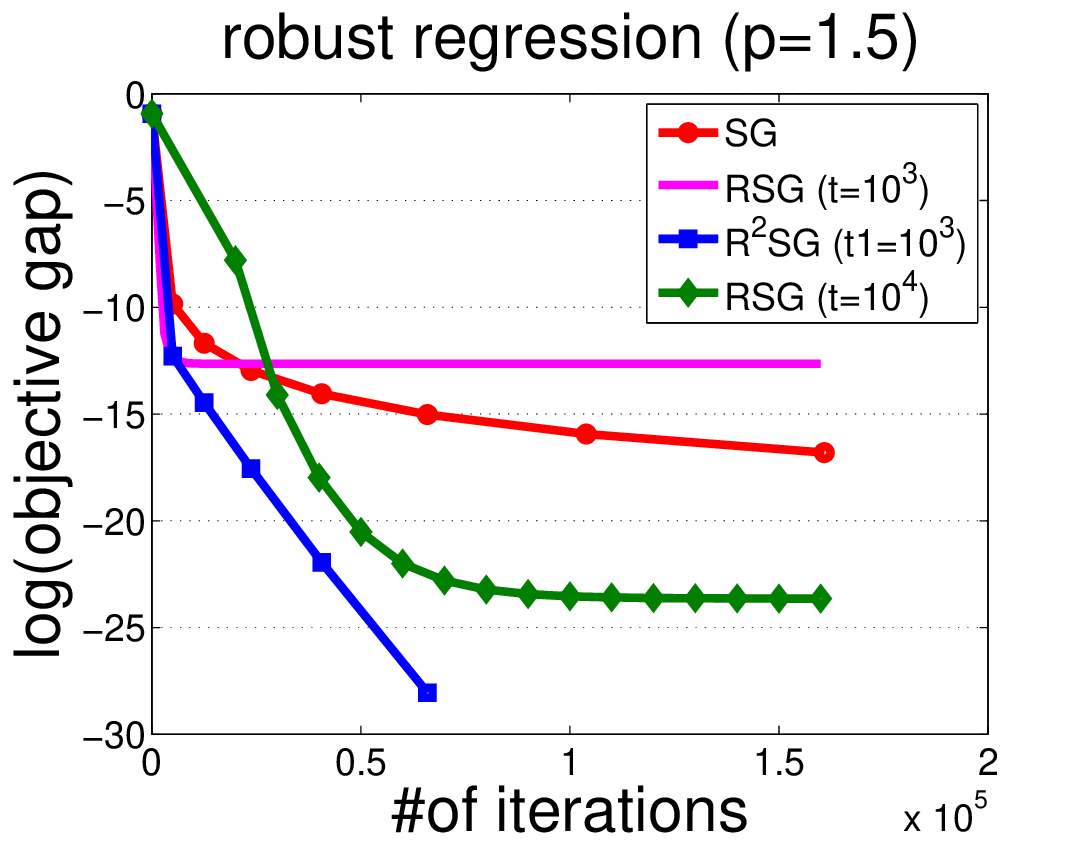}}\hspace*{0.1in}
\caption{Comparison of RSG with different $t$ and  of different algorithms on the space-ga data. One iteration means one subgradient update in all algorithms. }\label{fig:3}
\end{figure}

Secondly, we compare with SG to verify the effectiveness of RSG. The baseline SG is implemented with a decreasing step size proportional to $1/\sqrt{\tau}$, where $\tau$ is the iteration index. The initial step size of SG  is tuned in a wide range to  give the fastest convergence. The initial step size of RSG is  also tuned  around the best initial step size of SG. The results are shown in Figure~\ref{fig:4}(c,d) and Figure~\ref{fig:3}(c,d), where we show  RSG with two different values of $t$ and also R$^2$SG with an increasing sequence of $t$. In implementing  R$^2$SG, we restart RSG for every $5$ stages, and increase the number of iterations by a certain factor. In particular,we increase $t$ by a factor of $1.15$ and  $1.5$ respectively  for $p=1$ and $p=1.5$. From the results, we can see that (i) RSG with a smaller value of $t=10^3$ can quickly converge to an $\epsilon$-level, which is less accurate than SG after running a sufficiently large number of iterations; (ii) RSG with a relatively large value $t=10^4$ can converge to a much more accurate solution; (iv) R$^2$SG  converges much faster than SG and can bridge the gap between RSG-$t=10^3$ and RSG-$t=10^4$.

\begin{figure}[t]
\centering
\subfigure[comparison between different algorithms]{\includegraphics[scale=0.3]{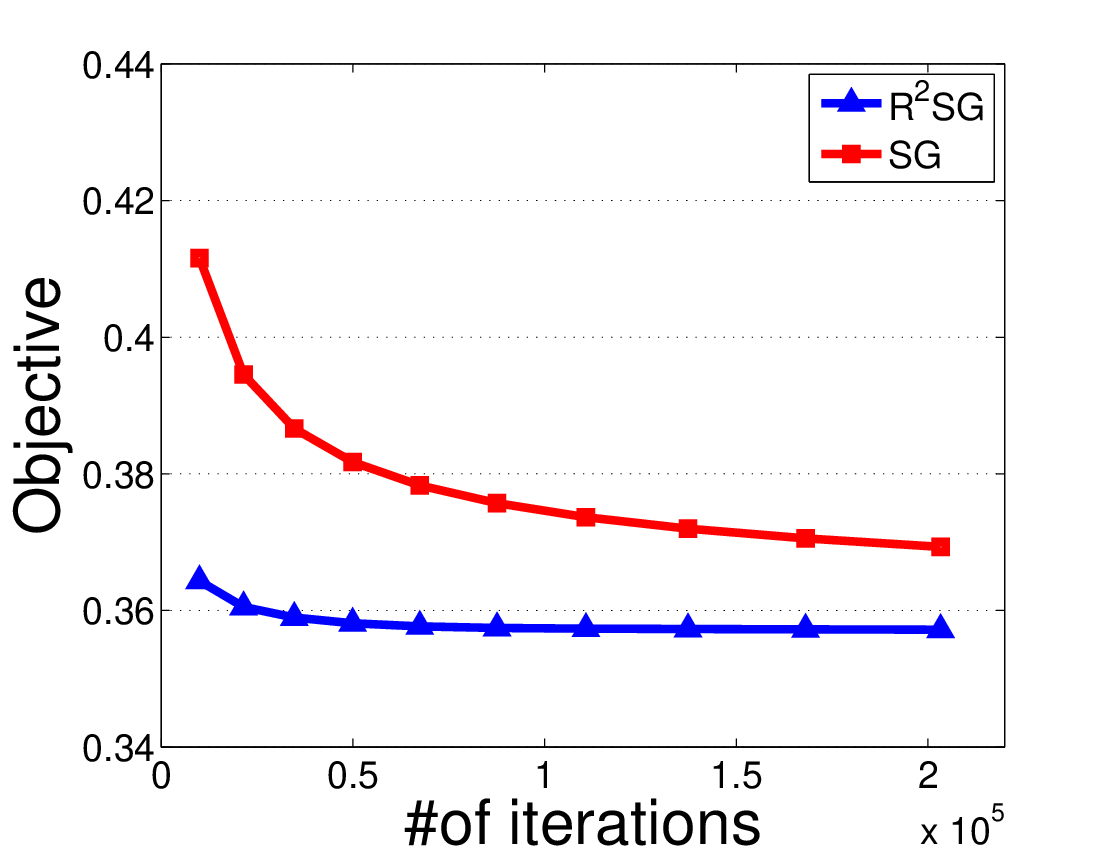}}\hspace*{0.2in}
\subfigure[sensitivity to initial solutions]{\includegraphics[scale=0.3]{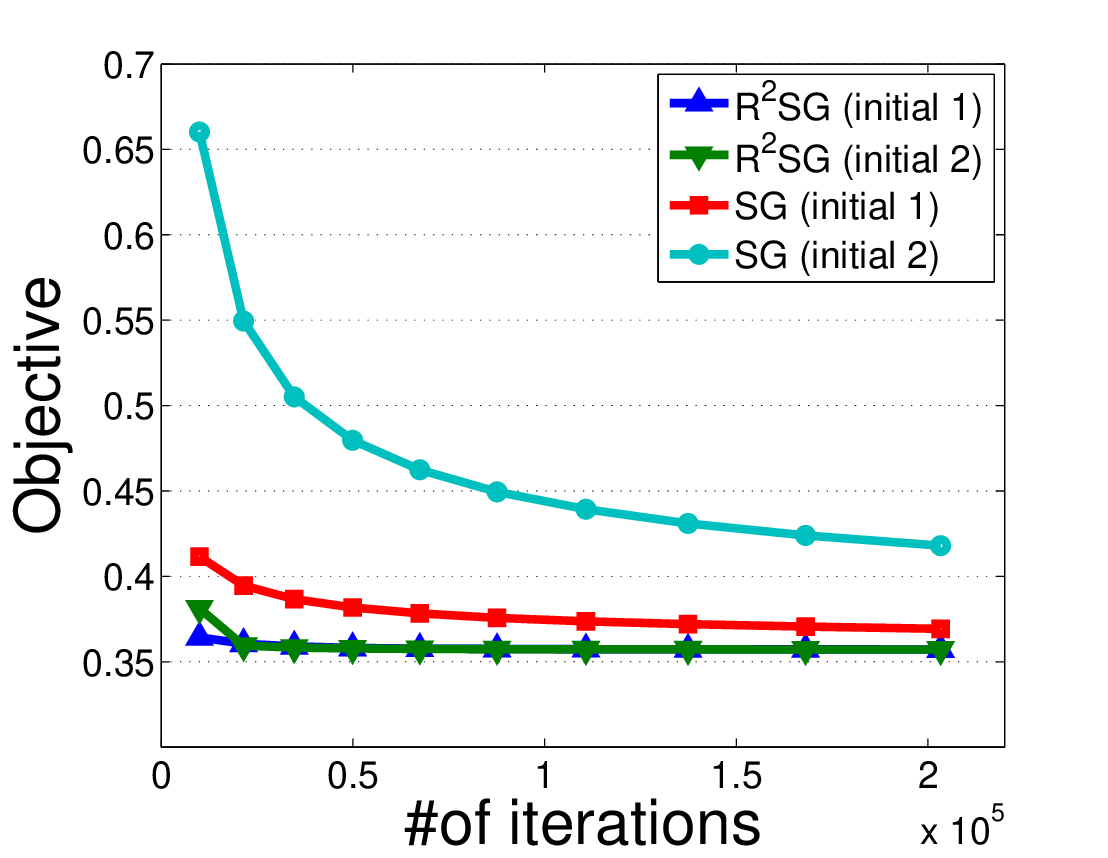}}
\caption{Results for solving SVM classification with GFlassso regularizer. In (b), the objective values of the two initial solutions are $1$ and $70.35$. One iteration means one subgradient update in all algorithms.}\label{fig:5}
\vspace*{-0.2in}
\end{figure}

\subsection{SVM Classification with a graph-guided fused lasso} The classification problem is to predict a binary class label $y\in\{1,-1\}$  based on a feature vector $\x\in\R^d$. Given a set of training examples $(\x_i,y_i), i=1,\ldots, n$, the problem of training a linear classification model can be cast into
\[
\min_{\w\in\R^d}F(\w):=\frac{1}{n}\sum_{i=1}^n\ell(\w^{\top}\x_i, y_i) + R(\w).
\]
Here we consider  the hinge loss $\ell(z, y) = \max(0, 1- yz)$ as in support vector machine (SVM) and a graph-guided fused lasso (GFlasso)  regularizer $R(\w) = \lambda \|F\w\|_1$~\citep{DBLP:journals/bioinformatics/KimSX09}, where $F=[F_{ij}]_{m\times d}\in\R^{m\times d}$ encodes the edge information between variables. Suppose there is a graph $\mathcal G=\{\mathcal V, \mathcal E\}$ where nodes $\mathcal V$ are the attributes and each edge is assigned a weight $s_{ij}$ that represents some kind of similarity between attribute $i$ and attribute $j$. Let $\mathcal E=\{e_1,\ldots, e_m\}$ denote a set of $m$ edges, where an edge $e_\tau=(i_\tau, j_\tau)$ consists of a tuple of two attributes. Then the $\tau$-th row of $F$ matrix can be formed by setting $F_{\tau, i_\tau}=s_{i_\tau, j_\tau}$ and $F_{\tau, j_\tau} = - s_{i_\tau, j_\tau}$ for $(i_\tau, j_\tau)\in\mathcal E$, and zeros for other entries. Then the GFlasso becomes $R(\w) = \lambda \sum_{(i,j)\in\mathcal E}s_{ij}|w_i - w_j|$. Previous studies have found that a carefully designed GFlasso regularization helps in reducing the risk of over-fitting. In this experiment, we follow~\citep{DBLP:conf/icml/OuyangHTG13} to generate a dependency graph by sparse inverse covariance selection~\citep{citeulike:2134265}.  To this end, we first generate a sparse  inverse covariance matrix using the method in~\citep{citeulike:2134265} and then assign an equal weight $s_{ij}=1$ to all edges that have non-zero entries in the resulting inverse covariance matrix.  We conduct the experiment on the dna data ($n=2000$ and $d=180$) from the libsvm website, which has three class labels. We solve the above problem to classify class 3 versus the rest. 
The comparison between different algorithms starting from an initial solution with all zero entries  for solving the above problem with $\lambda=0.1$ is presented in Figure~\ref{fig:5}(a). For R$^2$SG, we start from $t_1 = 10^3$ and restart it every $10$ stages with $t$ increased by a factor of $1.15$. The initial step sizes for all algorithms are tuned.

We also compare the dependence of  R$^2$SG's convergence on the initial solution with that of SG.  We use two different initial solutions (the first initial solution $\w_0=0$ and the second initial solution $\w_0$ is generated once from a normal Gaussian distribution). The convergence curves  of the two algorithms from the two different initial solutions are plotted in Figure~\ref{fig:5}(b). Note that the initial step sizes of SG and R$^2$SG are separately tuned for each initial solution. We can see that R$^2$SG is much  less sensitive to a bad initial solution than SG consistent with our theory.

{\subsection{Matrix Completion for Collaborative Filtering}

\begin{figure}[t]
\centering
\subfigure{\includegraphics[scale=0.3]{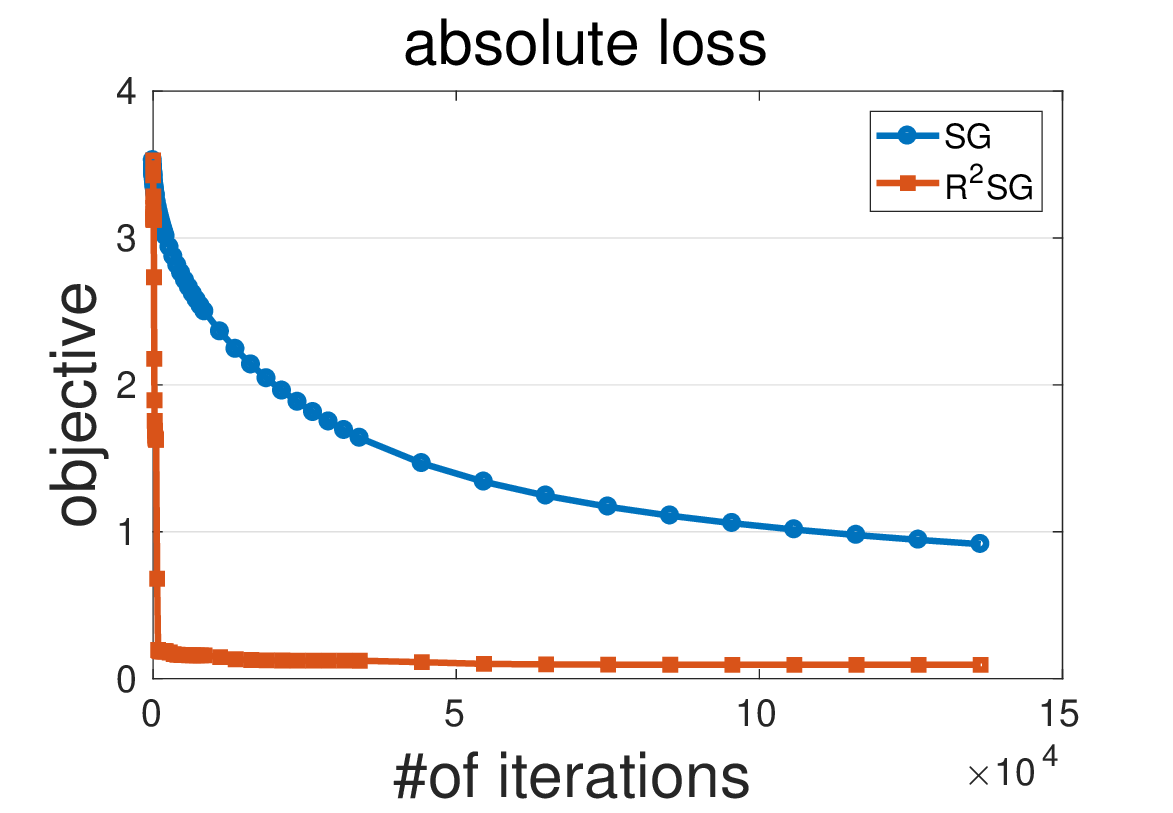}}\hspace*{0.2in}
\subfigure{\includegraphics[scale=0.3]{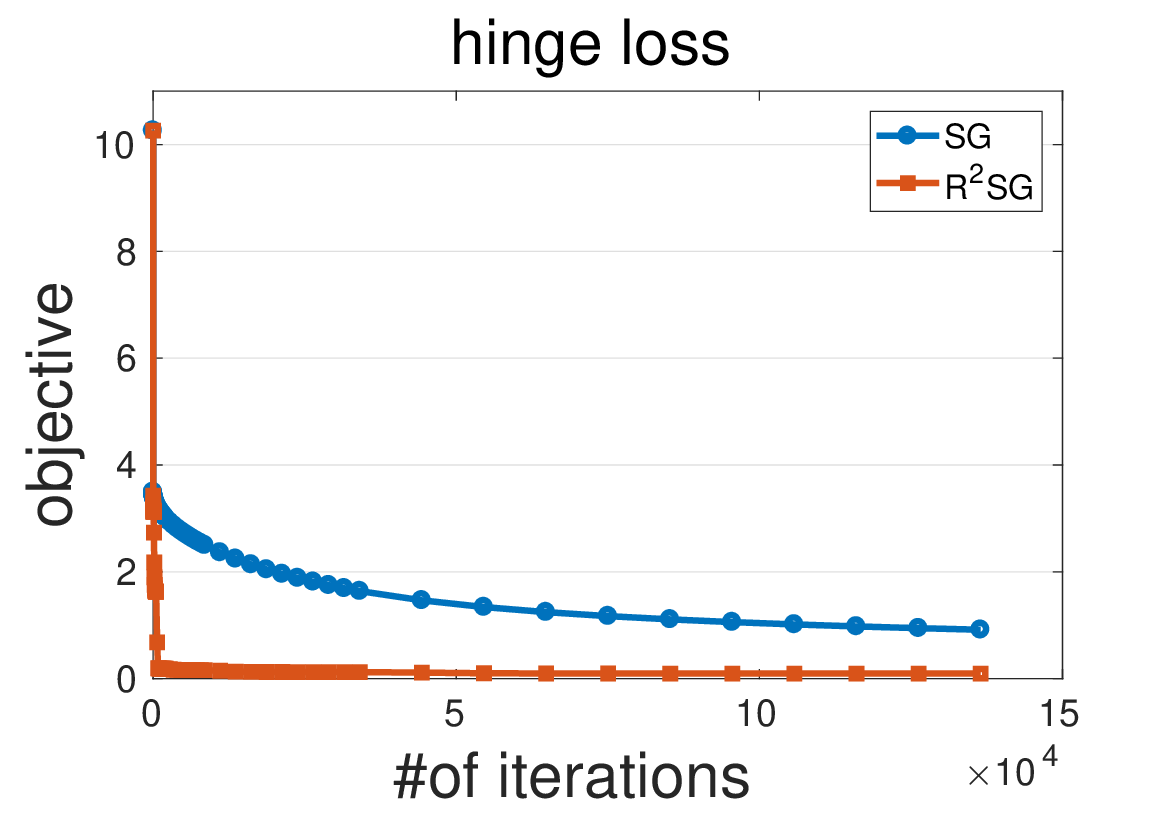}}
\caption{Results for solving low rank matrix completion with different loss functions.}\label{fig:mc}
\end{figure}

In this subsection, we consider low rank matrix completion problems to demonstrate the effectiveness of R$^2$SG without having  the knowledge of $c$ and $\theta$ in the local error bound condition. We consider a movie recommendation data set, namely MovieLens 100k data~\footnote{\url{https://grouplens.org/datasets/movielens/}}, which contains  $100,000$ ratings  from $m=943$ users on $n=1682$ movies.  We formulate the problem as a task of recovering a full  user-movie rating matrix $X$ from the partially observed matrix $Y$. The objective is
composed of a loss function measuring the difference between $X$ and $Y$ on the observed entries
and a nuclear norm regularizer on $X$ for enforcing a low rank, i.e., 
\begin{align}\label{eqn:mc}
\min_{X\in\R^{m\times n}} \frac{1}{N}\sum_{(i, j)\in\Sigma}\ell(X_{ij}, Y_{ij}) + \lambda \|X\|_*
\end{align}
where $\Sigma$ is a set of user-movie pairs that denote the observed entries, $\ell(\cdot, \cdot)$ denote a loss function, $\|X\|_*$ denotes the nuclear norm, $N=|\Sigma|$ and $\lambda>0$ is a regularization parameter. We consider two loss functions, i.e, the hinge loss and the absolute loss. For absolute loss, we set $\ell(a, b) = |a -b|$. For hinge loss, we follow that in~\citep{Rennie:2005:FMM:1102351.1102441} by introducing four thresholds $\theta_{1,2,3,4}$ due to  there
are five distinct ratings in $\{1, 2, 3, 4, 5\}$ that can be assigned to each movie, and defining 
$\ell(a, b) = \sum_{r=1}^{4}\max(0, 1 -T^r_{i,j}(\theta_r - X_{ij}))$, 
where $T^r_{ij}=\left\{\begin{array}{cc}1 &\text{ if } r\geq Y_{ij}\\ 0 &\text{ otherwise}\end{array}\right.$. In our experiment, we set $\theta_{1, 2, 3, 4}=(0, 3, 6, 9)$ and $\lambda=10^{-5}$ following~\citep{DBLP:journals/ml/YangMJZ14Non}. Since the loss function and the nuclear norm are both semi-algebraic functions~\citep{DBLP:journals/tnn/YangFS16,Bolte:2014:PAL:2650160.2650169}, then the problem~(\ref{eqn:mc}) satisfies an error bound condition on any compact set~\citep{arxiv:1510.08234}. However, it remains an open problem what are the proper values of $c$ and $\theta$ to make local error bound condition hold. Hence, we run R$^2$SG by setting $\theta=0$. To compare with SG, we simply set $t_1=10$ - the number of iterations of each stage of the first call of RSG.  The baseline SG is implemented in the same way as before. The results of the objective values vs the number of iterations  are plotted in Figure~\ref{fig:mc}. We can see that R$^2$SG converges much faster than SG, verifying the effectiveness of R$^2$SG predicted by  Theorem~\ref{thm:2RSG}.

}

\subsection{Comparison with Freund \& Lu's SG}\label{subsec:compFL}
In this subsection, we compare the proposed RSG with Freund \& Lu' SG algorithm empirically. The later algorithm is designed with a fixed relative accuracy $\epsilon'$ such that $\frac{f(\x_t) - f_*}{f_* - f_{slb}}\leq \epsilon'$, where $f_{slb}$ is a strict lower bound of $f_*$,  and requires to maintain the best solution in terms of the objective value during the optimization. For fair comparison,  we run RSG with a fixed $t$ and then vary $\epsilon'$ for Freund \& Lu's SG algorithm that is an input parameter, and then plot the objective values versus the running time and the number of iterations for both algorithms. The experiments are conducted on the two classification data sets as used in subsection~\ref{subsec:RR}, namely the housing data and the space-ga data, for solving robust regression problems~(\ref{eqn:rr}) with $p=1$ and $p=1.5$. The strict lower bound $f_{slb}$ in Freund \& Lu's algorithm is set to $0$.  The results are shown in Figure~\ref{fig:6} and Figure~\ref{fig:7}, where SGR refers to Freund \& Lu's SG algorithm with a specified relative accuracy. For each problem instance (a data set and a particular value of $p$), we report two results comparing the objective values  vs. running time and the number of iterations. 
  We can see that  RSG is very competitive in performance in terms of running time and converge faster than  Freund \& Lu's algorithm with a small  $\epsilon'=10^{-4}$ for achieving the same accurate solution (e.g., with objective gap less than $10^{-10}$). 
  

\begin{figure}[t]
\centering
\hspace*{-0.1in}\subfigure{\includegraphics[scale=0.19]{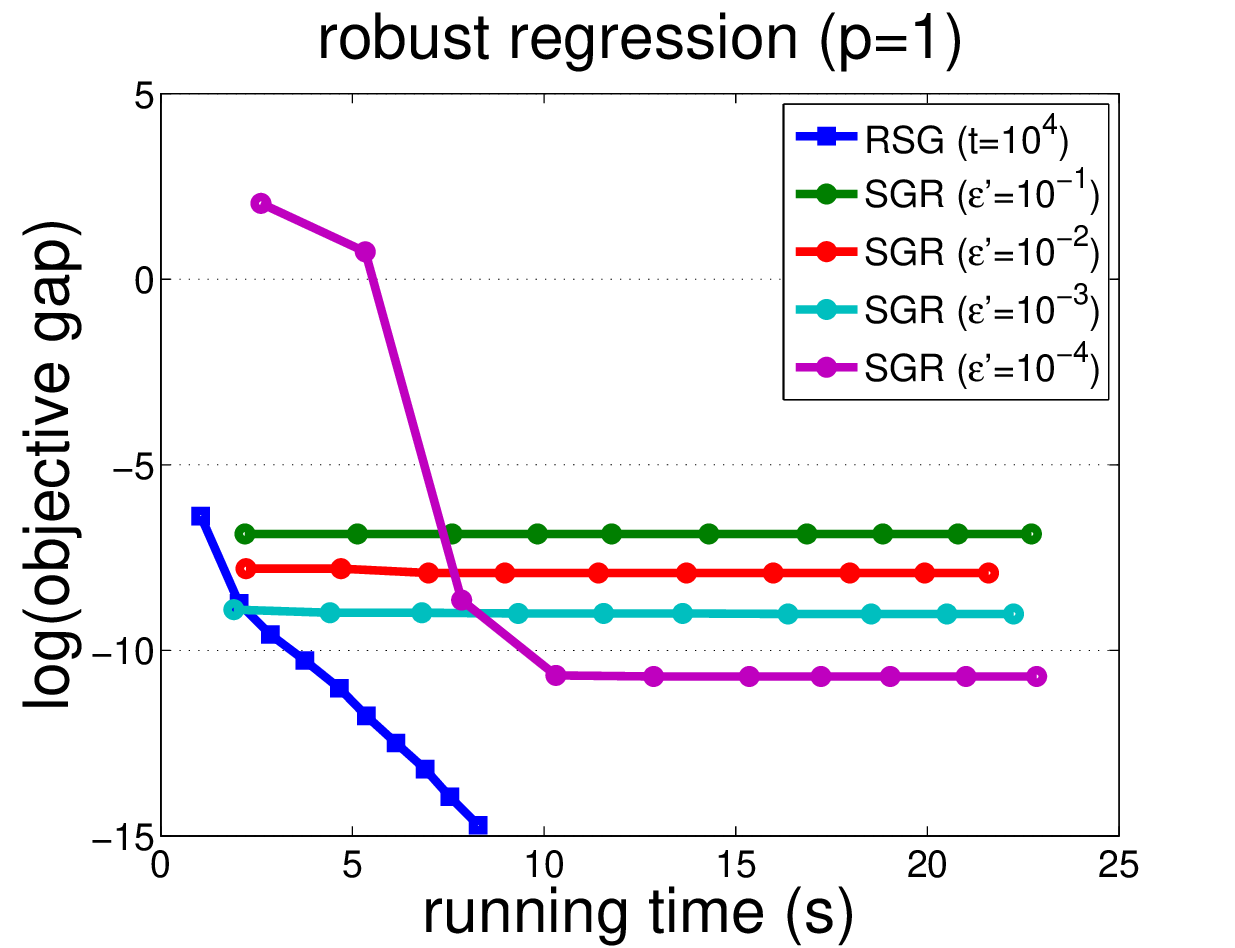}}\hspace*{-0.1in}
\subfigure{\includegraphics[scale=0.19]{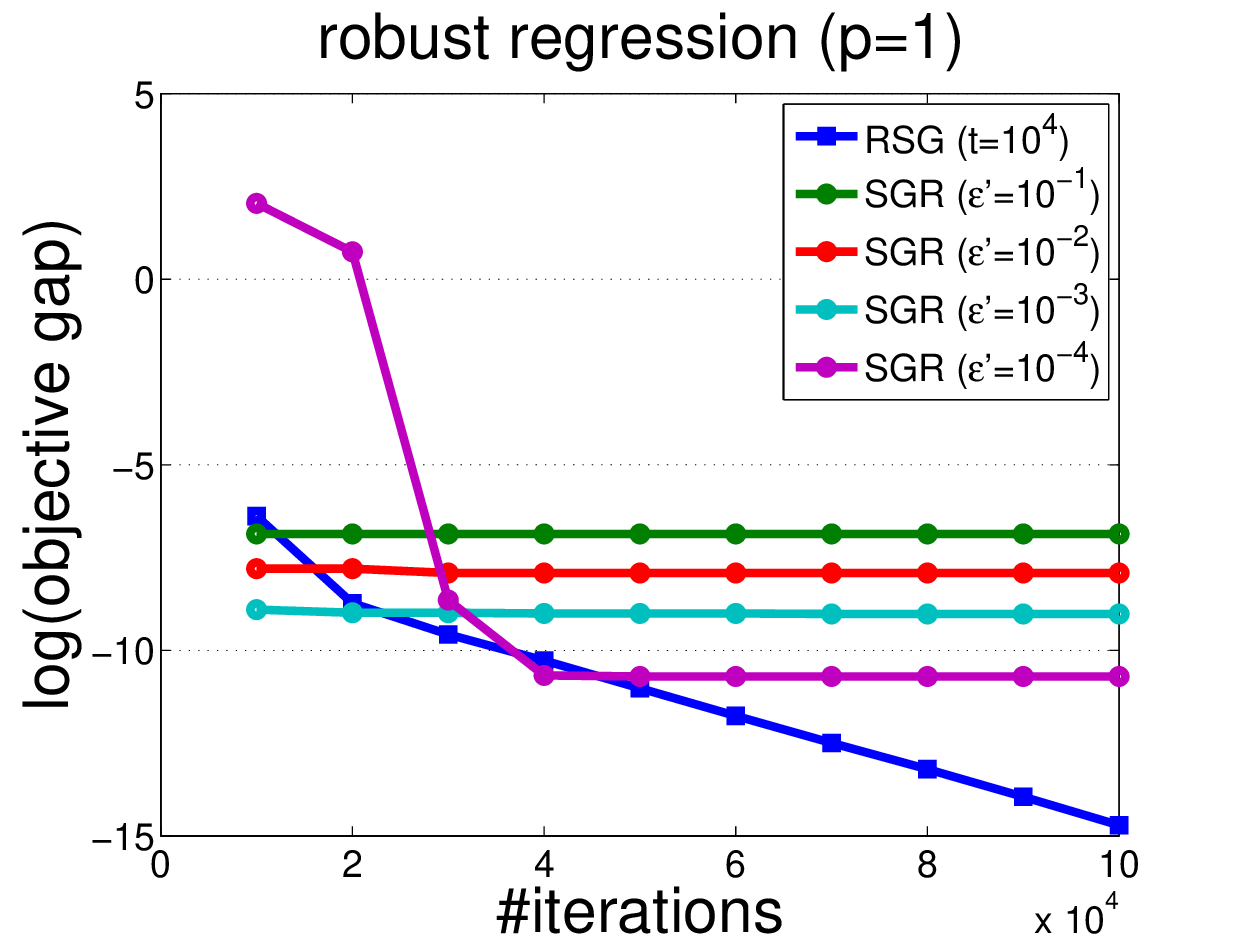}}\hspace*{-0.1in}
\subfigure{\includegraphics[scale=0.19]{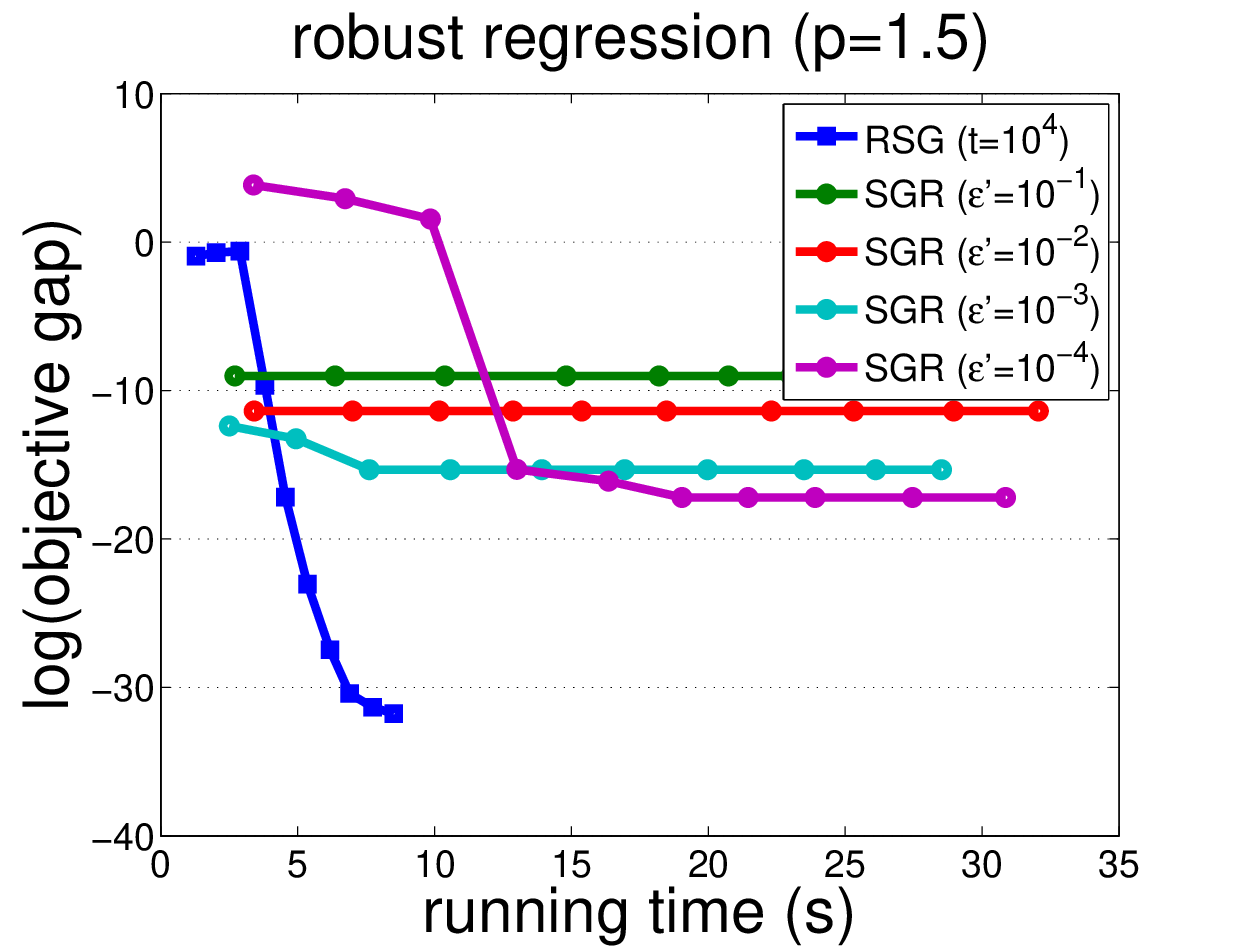}}\hspace*{-0.1in}
\subfigure{\includegraphics[scale=0.19]{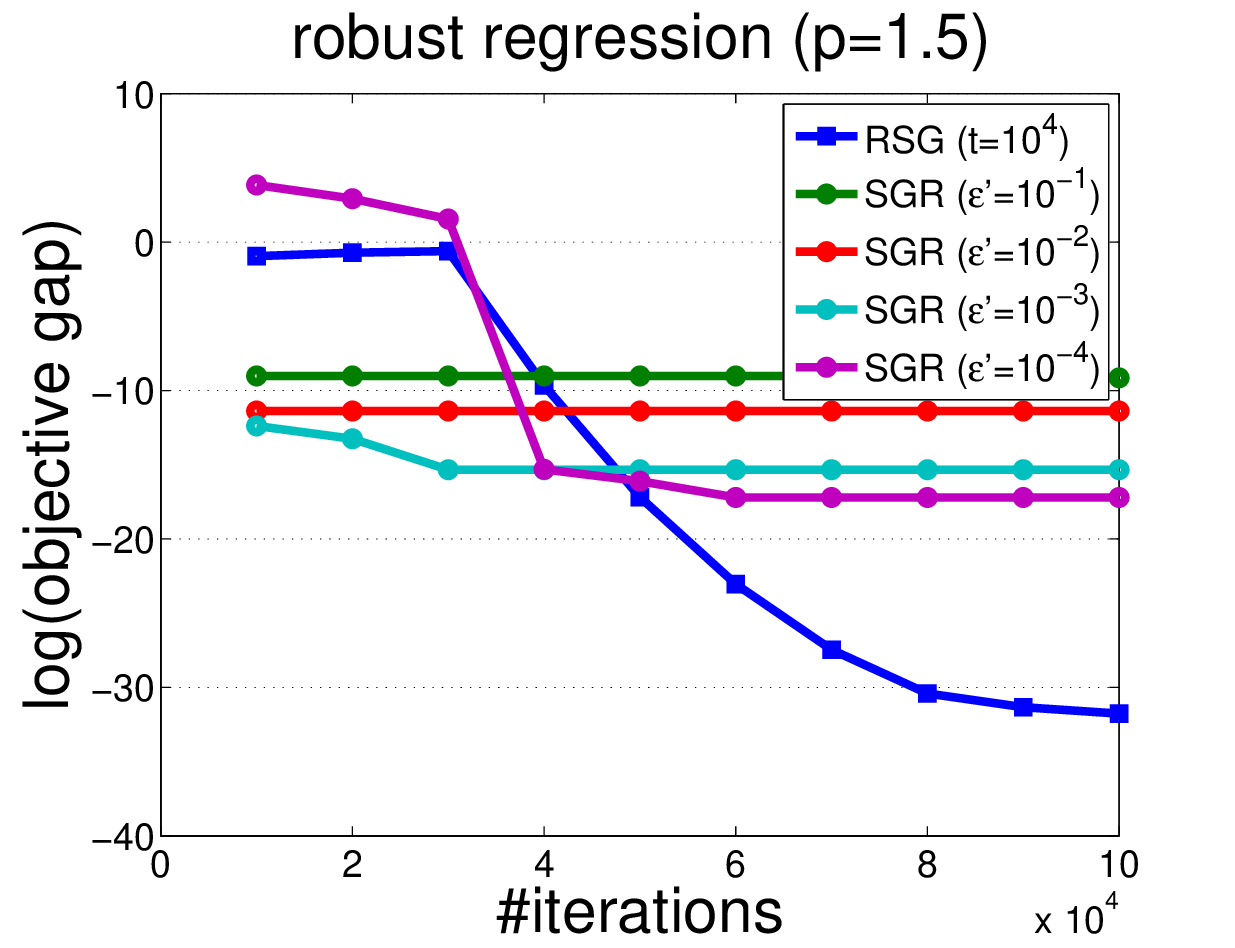}}\hspace*{-0.1in}
\caption{Comparison of RSG with Freund \& Lu's SG algorithm (SGR) on the housing data.  }\label{fig:6}

\centering
\hspace*{-0.1in}\subfigure{\includegraphics[scale=0.19]{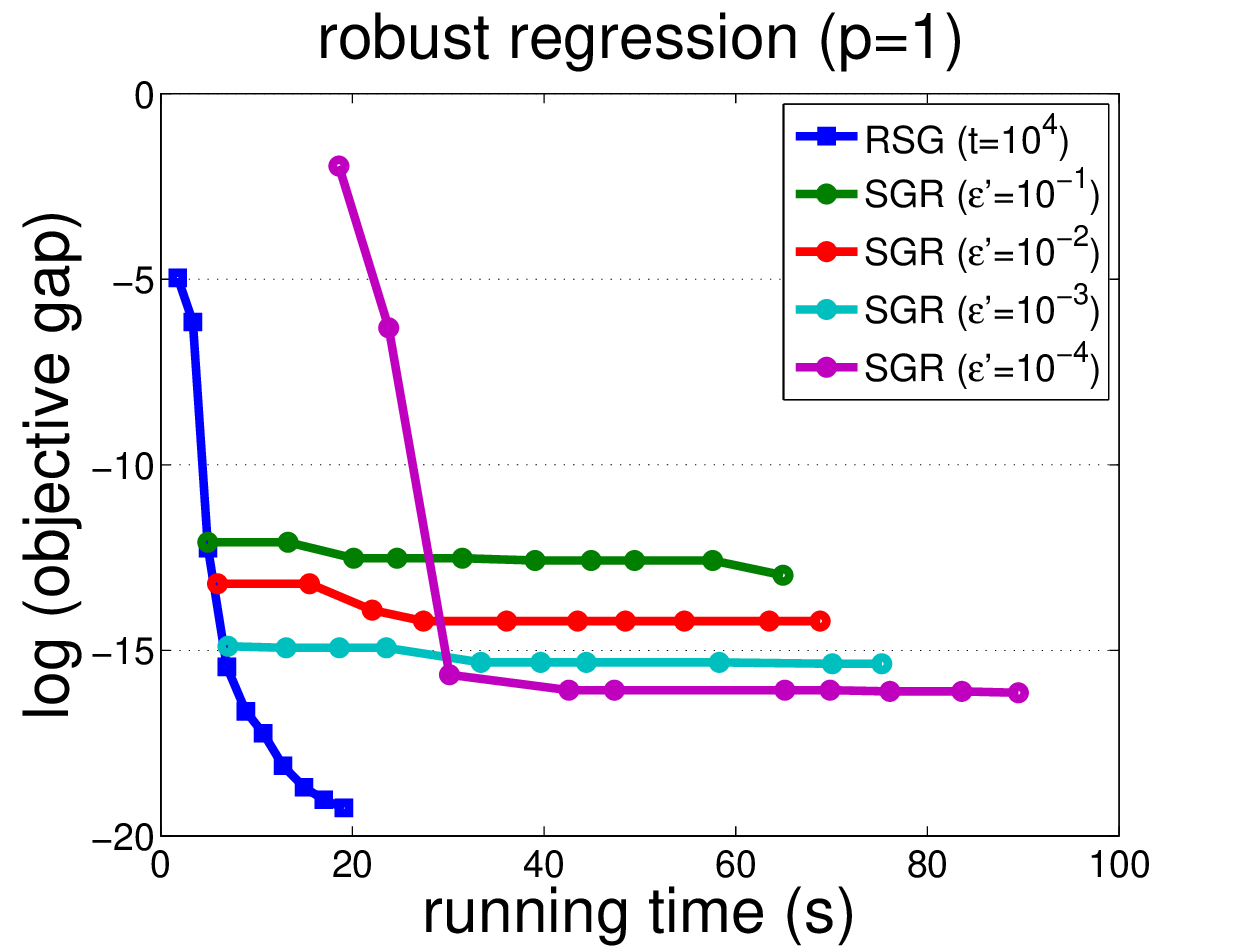}}\hspace*{-0.1in}
\subfigure{\includegraphics[scale=0.19]{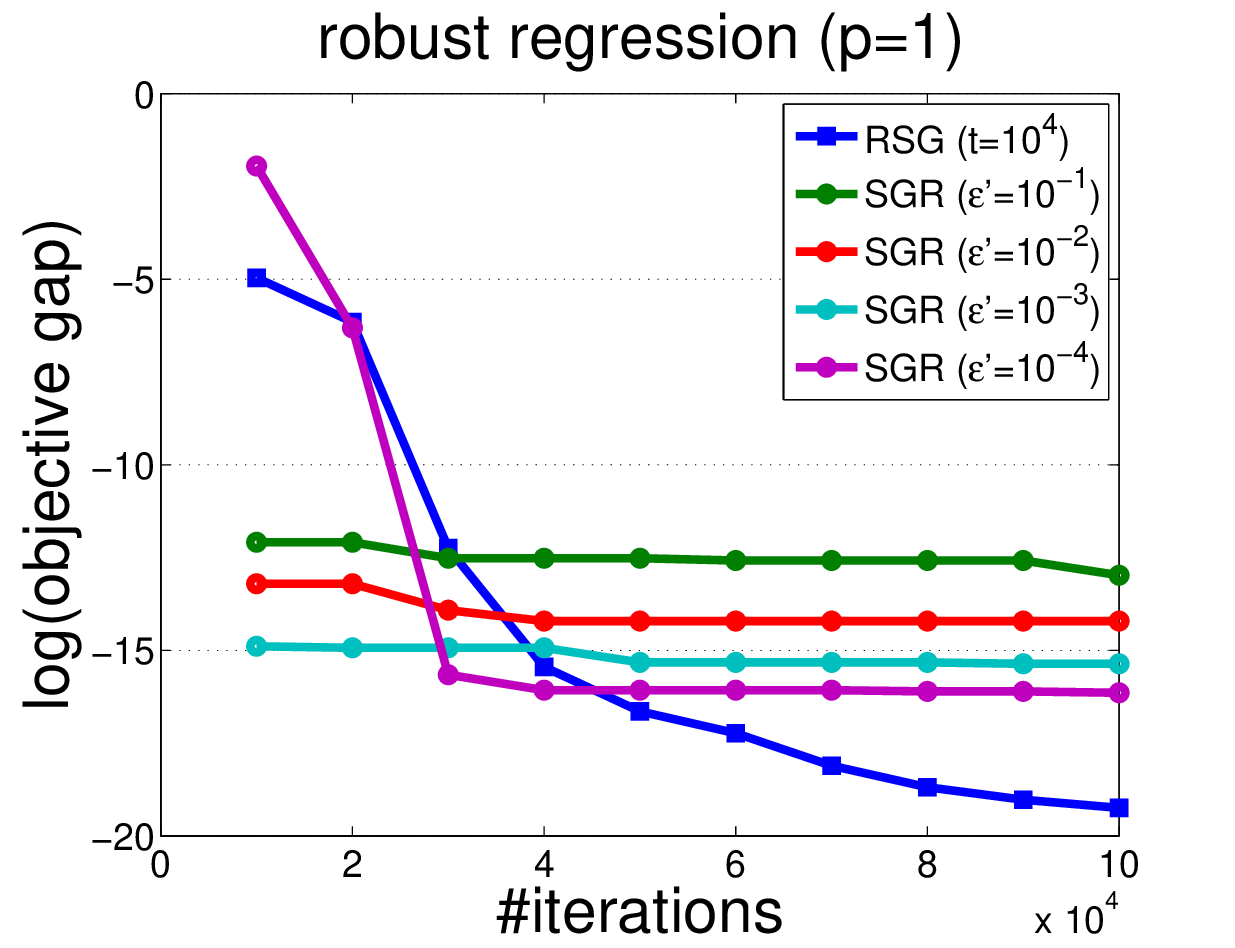}}\hspace*{-0.1in}
\subfigure{\includegraphics[scale=0.19]{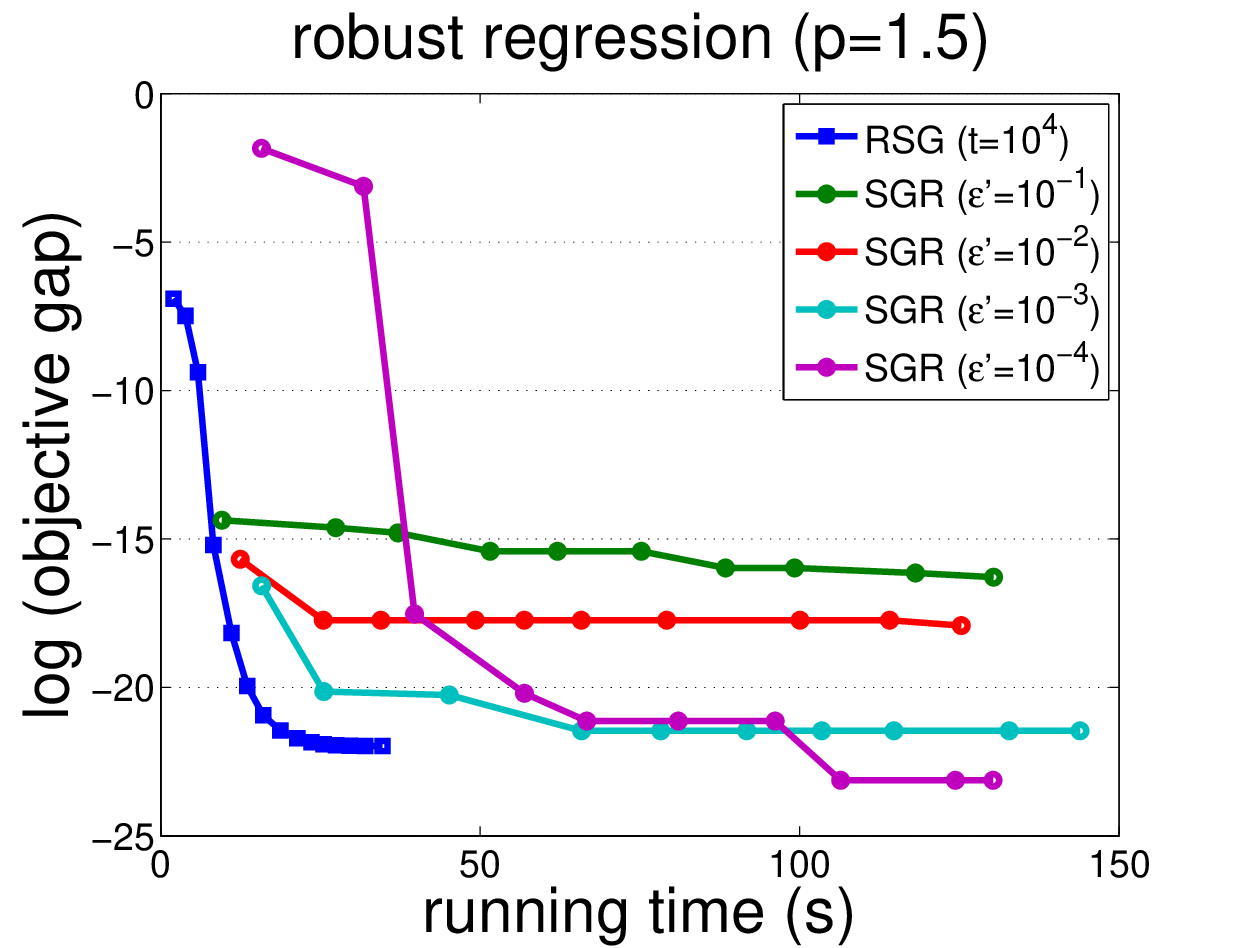}}\hspace*{-0.1in}
\subfigure{\includegraphics[scale=0.19]{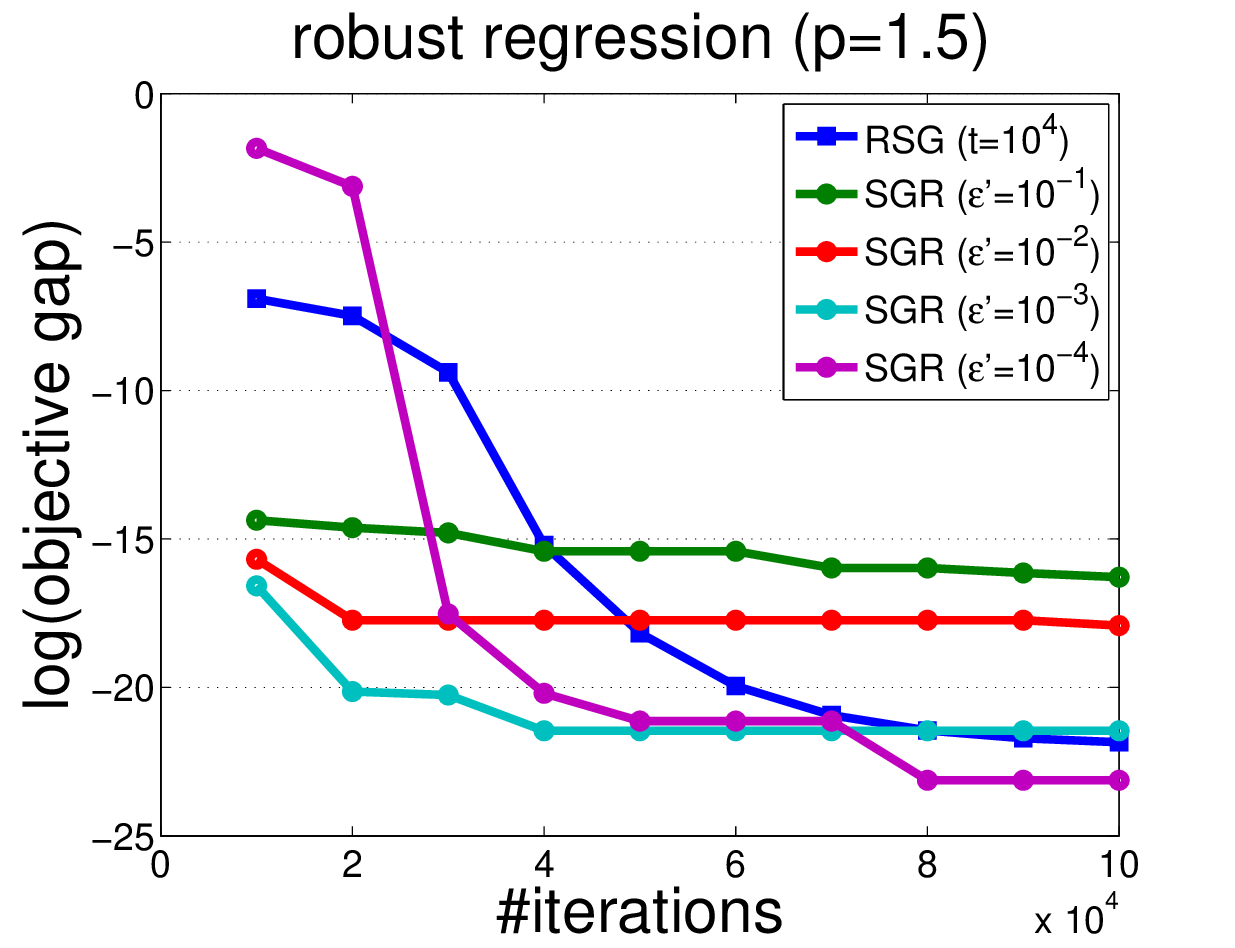}}\hspace*{-0.1in}
\caption{Comparison of RSG with Freund \& Lu's SG algorithm (SGR) on the space-ga data.  }\label{fig:7}
\end{figure}


\section{Conclusion}\label{sec:conc}

In this work, we have proposed a novel restarted subgradient method for  non-smooth and/or non-strongly convex optimization for obtaining an $\epsilon$-optimal solution. By leveraging the lower bound of the first-order optimality residual, we establish a generic complexity of RSG that improves over standard subgradient method. We have also considered several classes of non-smooth and non-strongly convex problems that admit a local error bound condition and derived the improved order of iteration complexities for RSG.  Several extensions have been made to design a parameter-free variant of RSG without requiring the knowledge of the constants in the local error bound condition. Experimental results on several machine learning tasks have demonstrated the effectiveness of the proposed algorithms in comparison to the subgradient method. 


\section*{Acknolwedgements}
We would like to sincerely thank anonymous reviewers for their very helpful comments. We thank James Renegar for pointing out the connection to his work and for his valuable comments on the difference between the two work. We also thank to Nghia T.A. Tran for pointing out the connection between the local error bound and metric subregularity of subdifferentials. Thanks to Mingrui Liu for spotting an error in the formulation of the $F$ matrix for GFlasso in earlier versions. 
T. Yang is supported by   NSF (1463988, 1545995).


\appendix

\section{A proposition needed to prove Corollary~\ref{thm:KL}}
The proof of Corollary~\ref{thm:KL} leverages the following result from~\citep{arxiv:1510.08234}.
\begin{prop}\label{prop:KL}\citep[Theorem 5]{arxiv:1510.08234}
Let $f(x)$ be an extended-valued, proper, convex and lower semicontinuous function that satisfies the KL inequality \eqref{eq:KLineq} at $x_*\in\arg\min f(\cdot)$ for all $x\in U\cap \{x: f(x_*)< f(x)< f(x_*)+\eta\}$, where $U$ is a neighborhood of $x_*$, then $dist(x, \arg\min f(\cdot))\leq \varphi(f(x) - f(x_*))$ for all $x\in U\cap \{x: f(x_*)< f(x)< f(x_*)+\eta\}$. 
\end{prop}


\bibliographystyle{abbrv}
\bibliography{all,lewa,icml11}

\begin{thebibliography}{76}
\providecommand{\natexlab}[1]{#1}
\providecommand{\url}[1]{\texttt{#1}}
\expandafter\ifx\csname urlstyle\endcsname\relax
  \providecommand{\doi}[1]{doi: #1}\else
  \providecommand{\doi}{doi: \begingroup \urlstyle{rm}\Url}\fi

\bibitem[Artacho and Geoffroy(2008)]{artacho:2008}
Francisco J~Arag�n Artacho and Michel~H Geoffroy.
\newblock Characterization of metric regularity of subdifferentials.
\newblock \emph{Journal of Convex Analysis}, 15:\penalty0 365--380, 2008.

\bibitem[Attouch et~al.(2010)Attouch, Bolte, Redont, and
  Soubeyran]{Attouch:2010:PAM:1836121.1836131}
H{\'e}dy Attouch, J{\'e}r\^{o}me Bolte, Patrick Redont, and Antoine Soubeyran.
\newblock Proximal alternating minimization and projection methods for
  nonconvex problems: An approach based on the kurdyka-lojasiewicz inequality.
\newblock \emph{Math. Oper. Res.}, 35:\penalty0 438--457, 2010.

\bibitem[Attouch et~al.(2013)Attouch, Bolte, and
  Svaiter]{journals/mp/AttouchBS13}
Hedy Attouch, J�r�me Bolte, and Benar~Fux Svaiter.
\newblock Convergence of descent methods for semi-algebraic and tame problems:
  proximal algorithms, forward-backward splitting, and regularized gauss-seidel
  methods.
\newblock \emph{Math. Program.}, 137\penalty0 (1-2):\penalty0 91--129, 2013.

\bibitem[Bach(2013)]{DBLP:journals/ftml/Bach13}
Francis~R. Bach.
\newblock Learning with submodular functions: {A} convex optimization
  perspective.
\newblock \emph{Foundations and Trends in Machine Learning}, 6\penalty0
  (2-3):\penalty0 145--373, 2013.

\bibitem[Bach and Moulines(2013)]{DBLP:conf/nips/BachM13}
Francis~R. Bach and Eric Moulines.
\newblock Non-strongly-convex smooth stochastic approximation with convergence
  rate o(1/n).
\newblock In \emph{Advances in Neural Information Processing Systems (NIPS)},
  pages 773--781, 2013.

\bibitem[Bertsimas and Copenhaver(2014)]{Bertsimas14characterizationof}
Dimitris Bertsimas and Martin~S. Copenhaver.
\newblock Characterization of the equivalence of robustification and
  regularization in linear, median, and matrix regression.
\newblock \emph{arXiv}, 2014.

\bibitem[Bolte et~al.(2006)Bolte, Daniilidis, and
  Lewis]{Bolte:2006:LIN:1328019.1328299}
J{\'e}r\^{o}me Bolte, Aris Daniilidis, and Adrian Lewis.
\newblock The {\l}ojasiewicz inequality for nonsmooth subanalytic functions
  with applications to subgradient dynamical systems.
\newblock \emph{SIAM J. on Optimization}, 17:\penalty0 1205--1223, 2006.

\bibitem[Bolte et~al.(2007)Bolte, Daniilidis, Lewis, and
  Shiota]{journals/siamjo/BolteDLS07}
J{\'e}r\^{o}me Bolte, Aris Daniilidis, Adrian~S. Lewis, and Masahiro Shiota.
\newblock Clarke subgradients of stratifiable functions.
\newblock \emph{SIAM Journal on Optimization}, 18:\penalty0 556--572, 2007.

\bibitem[Bolte et~al.(2014)Bolte, Sabach, and
  Teboulle]{Bolte:2014:PAL:2650160.2650169}
J{\'e}r\^{o}me Bolte, Shoham Sabach, and Marc Teboulle.
\newblock Proximal alternating linearized minimization for nonconvex and
  nonsmooth problems.
\newblock \emph{Mathematical Programming}, 146:\penalty0 459--494, 2014.

\bibitem[Bolte et~al.(2017)Bolte, Nguyen, Peypouquet, and
  Suter]{arxiv:1510.08234}
J{\'e}r{\^o}me Bolte, Trong~Phong Nguyen, Juan Peypouquet, and Bruce~W. Suter.
\newblock From error bounds to the complexity of first-order descent methods
  for convex functions.
\newblock \emph{Mathematical Programming}, 165\penalty0 (2):\penalty0 471--507,
  Oct 2017.
\newblock \doi{10.1007/s10107-016-1091-6}.

\bibitem[Burke and Deng(2002)]{DBLP:journals/mp/BurkeD02}
James~V. Burke and Sien Deng.
\newblock Weak sharp minima revisited part i: Basic theory.
\newblock \emph{Control and Cybernetics}, 31:\penalty0 439?469, 2002.

\bibitem[Burke and Deng(2005)]{DBLP:journals/mp/BurkeD05}
James~V. Burke and Sien Deng.
\newblock Weak sharp minima revisited, part {II:} application to linear
  regularity and error bounds.
\newblock \emph{Math. Program.}, 104\penalty0 (2-3):\penalty0 235--261, 2005.

\bibitem[Burke and Deng(2009)]{DBLP:journals/mp/BurkeD09}
James~V. Burke and Sien Deng.
\newblock Weak sharp minima revisited, part {III:} error bounds for
  differentiable convex inclusions.
\newblock \emph{Math. Program.}, 116\penalty0 (1-2):\penalty0 37--56, 2009.

\bibitem[Burke and Ferris.(1993)]{doi:10.1137/0331063}
James~V. Burke and Michael~C. Ferris.
\newblock Weak sharp minima in mathematical programming.
\newblock \emph{SIAM Journal on Control and Optimization}, 31\penalty0
  (5):\penalty0 1340--1359, 1993.
\newblock \doi{10.1137/0331063}.

\bibitem[Chen et~al.(2012)Chen, Lin, and Pena]{NIPS2012_4543}
Xi~Chen, Qihang Lin, and Javier Pena.
\newblock Optimal regularized dual averaging methods for stochastic
  optimization.
\newblock In \emph{Advances in Neural Information Processing Systems (NIPS)},
  pages 395--403. 2012.

\bibitem[Drusvyatskiy and Kempton(2016)]{Drusvyatskiy16b}
Dmitriy Drusvyatskiy and Courtney Kempton.
\newblock An accelerated algorithm for minimizing convex compositions.
\newblock \emph{arXiv:1605.00125}, 2016.

\bibitem[Drusvyatskiy and Lewis(2018)]{Drusvyatskiy16a}
Dmitriy Drusvyatskiy and Adrian~S. Lewis.
\newblock Error bounds, quadratic growth, and linear convergence of proximal
  methods.
\newblock \emph{Mathematics of Operations Research}, 2018.

\bibitem[Drusvyatskiy et~al.(2014)Drusvyatskiy, Mordukhovich, and
  Tran]{Drusvyatskiy14}
Dmitriy Drusvyatskiy, Boris Mordukhovich, and Nghia~T.A. Tran.
\newblock Second-order growth, tilt stability, and metric regularity of the
  subdifferential.
\newblock \emph{Journal of Convex Analysis}, 21:\penalty0 1165--1192, 2014.

\bibitem[Eremin(1965)]{eremin1965}
I.~I. Eremin.
\newblock The relaxation method of solving systems of inequalities with convex
  functions on the left-hand side.
\newblock \emph{Dokl. Akad. Nauk SSSR}, 160:\penalty0 994 -- 996, 1965.

\bibitem[Ferris(1991)]{Ferris1991}
Michael~C. Ferris.
\newblock Finite termination of the proximal point algorithm.
\newblock \emph{Mathematical Programming}, 50\penalty0 (1):\penalty0 359--366,
  Mar 1991.

\bibitem[Freund and Lu(2017)]{2015arXivRobert}
Robert~M. Freund and Haihao Lu.
\newblock New computational guarantees for solving convex optimization problems
  with first order methods, via a function growth condition measure.
\newblock \emph{Mathematical Programming}, 2017.

\bibitem[Friedman et~al.(2008)Friedman, Hastie, and
  Tibshirani]{citeulike:2134265}
Jerome Friedman, Trevor Hastie, and Robert Tibshirani.
\newblock Sparse inverse covariance estimation with the graphical lasso.
\newblock \emph{Biostatistics}, 9, 2008.

\bibitem[Ghadimi and Lan(2013)]{DBLP:journals/siamjo/GhadimiL13}
Saeed Ghadimi and Guanghui Lan.
\newblock Optimal stochastic approximation algorithms for strongly convex
  stochastic composite optimization, {II:} shrinking procedures and optimal
  algorithms.
\newblock \emph{{SIAM} Journal on Optimization}, 23\penalty0 (4):\penalty0
  2061--2089, 2013.

\bibitem[Gilpin et~al.(2012)Gilpin, Pe{\~{n}}a, and
  Sandholm]{DBLP:journals/mp/GilpinPS12}
Andrew Gilpin, Javier Pe{\~{n}}a, and Tuomas Sandholm.
\newblock First-order algorithm with log(1/epsilon) convergence for
  epsilon-equilibrium in two-person zero-sum games.
\newblock \emph{Math. Program.}, 133\penalty0 (1-2):\penalty0 279--298, 2012.

\bibitem[Goebel and Rockafellar(2007)]{Goebel_localstrong}
R.~Goebel and R.~T. Rockafellar.
\newblock Local strong convexity and local lipschitz continuity of the gradient
  of convex functions.
\newblock \emph{Journal of Convex Analysis}, 2007.

\bibitem[Gong and Ye(2014)]{DBLP:journals/corr/GongY14}
Pinghua Gong and Jieping Ye.
\newblock Linear convergence of variance-reduced projected stochastic gradient
  without strong convexity.
\newblock \emph{CoRR}, abs/1406.1102, 2014.

\bibitem[Hazan and Kale(2011)]{hazan-20110-beyond}
Elad Hazan and Satyen Kale.
\newblock Beyond the regret minimization barrier: an optimal algorithm for
  stochastic strongly-convex optimization.
\newblock In \emph{Proceedings of the 24th Annual Conference on Learning Theory
  (COLT)}, pages 421--436, 2011.

\bibitem[Hou et~al.(2013)Hou, Zhou, So, and Luo]{DBLP:conf/nips/HouZSL13}
Ke~Hou, Zirui Zhou, Anthony~Man{-}Cho So, and Zhi{-}Quan Luo.
\newblock On the linear convergence of the proximal gradient method for trace
  norm regularization.
\newblock In \emph{Advances in Neural Information Processing Systems (NIPS)},
  pages 710--718, 2013.

\bibitem[Juditsky and Nesterov(2014)]{Nesterov:2014:uniform_convex}
Anatoli Juditsky and Yuri Nesterov.
\newblock Deterministic and stochastic primal-dual subgradient algorithms for
  uniformly convex minimization.
\newblock \emph{Stochastic Systems}, 4\penalty0 (1):\penalty0 44--80, 2014.

\bibitem[Karimi et~al.(2016)Karimi, Nutini, and
  Schmidt]{DBLP:conf/pkdd/KarimiNS16}
Hamed Karimi, Julie Nutini, and Mark~W. Schmidt.
\newblock Linear convergence of gradient and proximal-gradient methods under
  the polyak-{\l}ojasiewicz condition.
\newblock In \emph{Machine Learning and Knowledge Discovery in Databases -
  European Conference (ECML-PKDD)}, pages 795--811, 2016.

\bibitem[Kim et~al.(2009)Kim, Sohn, and
  Xing]{DBLP:journals/bioinformatics/KimSX09}
Seyoung Kim, Kyung{-}Ah Sohn, and Eric~P. Xing.
\newblock A multivariate regression approach to association analysis of a
  quantitative trait network.
\newblock \emph{Bioinformatics}, 25\penalty0 (12), 2009.

\bibitem[Kruger(2015)]{kruger2015}
Alexander~Y. Kruger.
\newblock Error bounds and h\"{o}lder metric subregularity.
\newblock \emph{Set-Valued and Variational Analysis}, 23:\penalty0 705--736,
  2015.

\bibitem[Lacoste{-}Julien et~al.(2012)Lacoste{-}Julien, Schmidt, and
  Bach]{DBLP:journals/corr/abs-1212-2002}
Simon Lacoste{-}Julien, Mark~W. Schmidt, and Francis~R. Bach.
\newblock A simpler approach to obtaining an o(1/t) convergence rate for the
  projected stochastic subgradient method.
\newblock \emph{CoRR}, abs/1212.2002, 2012.
\newblock URL \url{http://arxiv.org/abs/1212.2002}.

\bibitem[Li(2010)]{DBLP:journals/siamjo/Li10}
Guoyin Li.
\newblock On the asymptotically well behaved functions and global error bound
  for convex polynomials.
\newblock \emph{{SIAM} Journal on Optimization}, 20\penalty0 (4):\penalty0
  1923--1943, 2010.

\bibitem[Li(2013)]{DBLP:journals/mp/Li13}
Guoyin Li.
\newblock Global error bounds for piecewise convex polynomials.
\newblock \emph{Math. Program.}, 137\penalty0 (1-2):\penalty0 37--64, 2013.

\bibitem[Li and Mordukhovich(2012)]{DBLP:journals/siamjo/LiM12}
Guoyin Li and Boris~S. Mordukhovich.
\newblock H{\"{o}}lder metric subregularity with applications to proximal point
  method.
\newblock \emph{{SIAM} Journal on Optimization}, 22\penalty0 (4):\penalty0
  1655--1684, 2012.

\bibitem[Luo and Tseng(1992{\natexlab{a}})]{Luo:1992a}
Zhi-Quan Luo and Paul Tseng.
\newblock On the convergence of coordinate descent method for convex
  differentiable minization.
\newblock \emph{Journal of Optimization Theory and Applications}, 72\penalty0
  (1):\penalty0 7--35, 1992{\natexlab{a}}.

\bibitem[Luo and Tseng(1992{\natexlab{b}})]{Luo:1992b}
Zhi-Quan Luo and Paul Tseng.
\newblock On the linear convergence of descent methods for convex essenially
  smooth minization.
\newblock \emph{SIAM Journal on Control and Optimization}, 30\penalty0
  (2):\penalty0 408--425, 1992{\natexlab{b}}.

\bibitem[Luo and Tseng(1993)]{Luo:1993}
Zhi-Quan Luo and Paul Tseng.
\newblock Error bounds and convergence analysis of feasible descent methods: a
  general approach.
\newblock \emph{Annals of Operations Research}, 46:\penalty0 157--178, 1993.

\bibitem[Mordukhovich and
  Ouyang(2015)]{RePEc:spr:jglopt:v:63:y:2015:i:4:p:777-795}
Boris Mordukhovich and Wei Ouyang.
\newblock Higher-order metric subregularity and its applications.
\newblock \emph{Journal of Global Optimization- An International Journal
  Dealing with Theoretical and Computational Aspects of Seeking Global Optima
  and Their Applications in Science, Management and Engineering}, 63\penalty0
  (4):\penalty0 777--795, 2015.

\bibitem[Necoara and Clipici(2016)]{doi:10.1137/130950288}
Ion Necoara and Dragos Clipici.
\newblock Parallel random coordinate descent method for composite minimization:
  Convergence analysis and error bounds.
\newblock \emph{SIAM Journal on Optimization}, 26\penalty0 (1):\penalty0
  197--226, 2016.
\newblock \doi{10.1137/130950288}.

\bibitem[Necoara et~al.(2015)Necoara, Nesterov, and
  Glineur]{DBLP:journals/corr/nesterov16linearnon}
Ion Necoara, Yurri Nesterov, and Francois Glineur.
\newblock Linear convergence of first order methods for non-strongly convex
  optimization.
\newblock \emph{CoRR}, abs/1504.06298, 2015.

\bibitem[Nemirovski et~al.(2009)Nemirovski, Juditsky, Lan, and
  Shapiro]{Nemirovski:2009:RSA:1654243.1654247}
Arkadi Nemirovski, Anatoli Juditsky, Guanghui Lan, and Alexander Shapiro.
\newblock Robust stochastic approximation approach to stochastic programming.
\newblock \emph{SIAM Journal on Optimization}, 19:\penalty0 1574--1609, 2009.

\bibitem[Nemirovsky~A.S. and Yudin(1983)]{opac-b1091338}
Arkadii~Semenovich. Nemirovsky~A.S. and D.~B Yudin.
\newblock \emph{Problem complexity and method efficiency in optimization}.
\newblock Wiley-Interscience series in discrete mathematics. Wiley, Chichester,
  New York, 1983.
\newblock ISBN 0-471-10345-4.
\newblock A Wiley-Interscience publication.

\bibitem[Nesterov(2004)]{opac-b1104789}
Yurii Nesterov.
\newblock \emph{Introductory lectures on convex optimization : a basic course}.
\newblock Applied optimization. Kluwer Academic Publ., 2004.
\newblock ISBN 1-4020-7553-7.

\bibitem[Nesterov(2005)]{Nesterov:2005:SMN}
Yurii Nesterov.
\newblock Smooth minimization of non-smooth functions.
\newblock \emph{Mathematical Programming}, 103\penalty0 (1):\penalty0 127--152,
  2005.

\bibitem[Nesterov(2009)]{Nesterov:2009:PSM:1530733.1530741}
Yurii Nesterov.
\newblock Primal-dual subgradient methods for convex problems.
\newblock \emph{Mathematical Programming}, 120:\penalty0 221--259, 2009.

\bibitem[Ouyang et~al.(2013)Ouyang, He, Tran, and
  Gray]{DBLP:conf/icml/OuyangHTG13}
Hua Ouyang, Niao He, Long Tran, and Alexander~G. Gray.
\newblock Stochastic alternating direction method of multipliers.
\newblock In \emph{Proceedings of the 30th International Conference on Machine
  Learning (ICML)}, pages 80--88, 2013.

\bibitem[Pang(1987)]{Pang:1987:PEB:35577.35584}
Jong-Shi Pang.
\newblock A posteriori error bounds for the linearly-constrained variational
  inequality problem.
\newblock \emph{Mathematics of Operations Research}, 12\penalty0 (3):\penalty0
  474--484, August 1987.
\newblock ISSN 0364-765X.

\bibitem[Pang(1997)]{Pang:1997}
Jong-Shi Pang.
\newblock Error bounds in mathematical programming.
\newblock \emph{Mathematical Programming}, 79\penalty0 (1):\penalty0 299--332,
  October 1997.

\bibitem[Polyak(1979)]{polysharp1979}
B.~T. Polyak.
\newblock Sharp minima.
\newblock In \emph{Proceedings of the IIASA Workshop on Generalized Lagrangians
  and Their Applications}, Institute of Control Sciences Lecture Notes, Moscow.
  1979.

\bibitem[Polyak(1987)]{citeulike:13796896}
B.~T. Polyak.
\newblock \emph{Introduction to Optimization}.
\newblock Optimization Software Inc, New York, 1987.

\bibitem[Polyak(1969)]{polyak1969}
B.T. Polyak.
\newblock Minimization of unsmooth functionals.
\newblock \emph{USSR Computational Mathematics and Mathematical Physics},
  9:\penalty0 509 -- 521, 1969.

\bibitem[Renegar(2014)]{2014arXivJames}
James Renegar.
\newblock Efficient first-order methods for linear programming and semidefinite
  programming.
\newblock \emph{ArXiv e-prints}, 2014.

\bibitem[Renegar(2015)]{2015arXivJames}
James Renegar.
\newblock A framework for applying subgradient methods to conic optimization
  problems.
\newblock \emph{ArXiv e-prints}, 2015.

\bibitem[Renegar(2016)]{2016arXivJames}
James Renegar.
\newblock ``efficient'' subgradient methods for general convex optimization.
\newblock \emph{SIAM Journal on Optimization}, 26\penalty0 (4):\penalty0
  2649--2676, 2016.

\bibitem[Rennie and Srebro(2005)]{Rennie:2005:FMM:1102351.1102441}
Jasson D.~M. Rennie and Nathan Srebro.
\newblock Fast maximum margin matrix factorization for collaborative
  prediction.
\newblock In \emph{Proceedings of the 22Nd International Conference on Machine
  Learning}, pages 713--719, New York, NY, USA, 2005. ACM.
\newblock ISBN 1-59593-180-5.
\newblock \doi{10.1145/1102351.1102441}.
\newblock URL \url{http://doi.acm.org/10.1145/1102351.1102441}.

\bibitem[Rockafellar(1970)]{rockafellar1970convex}
R.T. Rockafellar.
\newblock \emph{Convex Analysis}.
\newblock Princeton mathematical series. Princeton University Press, 1970.

\bibitem[Schneider and Uschmajew(2015)]{DBLP:journals/corr/SchneiderU14}
Reinhold Schneider and Andr{\'{e}} Uschmajew.
\newblock Convergence results for projected line-search methods on varieties of
  low-rank matrices via {\l}ojasiewicz inequality.
\newblock \emph{SIAM Journal on Optimization}, 25\penalty0 (1):\penalty0
  622--646, 2015.

\bibitem[So and Zhou(2017)]{DBLP:journals/corr/So13}
Anthony~Man{-}Cho So and Zirui Zhou.
\newblock Non-asymptotic convergence analysis of inexact gradient methods for
  machine learning without strong convexity.
\newblock \emph{Optimization Methods and Software}, 32:\penalty0 963 -- 992,
  2017.

\bibitem[Studniarski and Ward(1999)]{doi:10.1137/S0363012996301269}
Marcin Studniarski and Doug~E. Ward.
\newblock Weak sharp minima: Characterizations and sufficient conditions.
\newblock \emph{SIAM Journal on Control and Optimization}, 38\penalty0
  (1):\penalty0 219--236, 1999.
\newblock \doi{10.1137/S0363012996301269}.

\bibitem[Tseng and S.(2009{\natexlab{a}})]{TsengYun:2009}
P.~Tseng and Yun S.
\newblock A block-coordinate gradient descent method for linearly constrained
  nonsmooth separable optimization.
\newblock \emph{Journal of Optimization Theory Application}, 140:\penalty0
  513--535, 2009{\natexlab{a}}.

\bibitem[Tseng and S.(2009{\natexlab{b}})]{TsengYun:2010}
P.~Tseng and Yun S.
\newblock A coordinate gradient descent method for nonsmooth separable
  minimization.
\newblock \emph{Mathematical Programming}, 117:\penalty0 387--423,
  2009{\natexlab{b}}.

\bibitem[Wang and Lin(2014)]{DBLP:journals/jmlr/WangL14}
Po{-}Wei Wang and Chih{-}Jen Lin.
\newblock Iteration complexity of feasible descent methods for convex
  optimization.
\newblock \emph{Journal of Machine Learning Research}, 15\penalty0
  (1):\penalty0 1523--1548, 2014.

\bibitem[Xu et~al.(2010)Xu, Caramanis, and Mannor]{DBLP:journals/tit/XuCM10}
Huan Xu, Constantine Caramanis, and Shie Mannor.
\newblock Robust regression and lasso.
\newblock \emph{{IEEE} Trans. Information Theory}, 56\penalty0 (7):\penalty0
  3561--3574, 2010.

\bibitem[Xu et~al.(2016)Xu, Yan, Lin, and
  Yang]{DBLP:journals/corr/abs-1607-03815}
Yi~Xu, Yan Yan, Qihang Lin, and Tianbao Yang.
\newblock Homotopy smoothing for non-smooth problems with lower complexity than
  1/epsilon.
\newblock In \emph{Advances in Neural Information Processing Systems}, 2016.

\bibitem[Xu et~al.(2017)Xu, Lin, and Yang]{DBLP:journals/corr/abs-1607-01027}
Yi~Xu, Qihang Lin, and Tianbao Yang.
\newblock Stochastic convex optimization: Faster local growth implies faster
  global convergence.
\newblock In \emph{Proceedings of the 34th International Conference on Machine
  Learning, (ICML)}, pages 3821--3830, 2017.

\bibitem[Yang and Lin(2016)]{DBLP:journals/corr/arXiv:1512.03107}
Tianbao Yang and Qihang Lin.
\newblock Rsg: Beating sgd without smoothness and/or strong convexity.
\newblock \emph{CoRR}, abs/1512.03107, 2016.

\bibitem[Yang et~al.(2014)Yang, Mahdavi, Jin, and
  Zhu]{DBLP:journals/ml/YangMJZ14Non}
Tianbao Yang, Mehrdad Mahdavi, Rong Jin, and Shenghuo Zhu.
\newblock An efficient primal-dual prox method for non-smooth optimization.
\newblock \emph{Machine Learning}, 2014.

\bibitem[Yang(2009)]{doi:10.1137/070689838}
W.~H. Yang.
\newblock Error bounds for convex polynomials.
\newblock \emph{SIAM Journal on Optimization}, 19\penalty0 (4):\penalty0
  1633--1647, 2009.

\bibitem[Yang et~al.(2016)Yang, Feng, and Suykens]{DBLP:journals/tnn/YangFS16}
Yuning Yang, Yunlong Feng, and Johan A.~K. Suykens.
\newblock Robust low-rank tensor recovery with regularized redescending
  m-estimator.
\newblock \emph{{IEEE} Trans. Neural Netw. Learning Syst.}, 27\penalty0
  (9):\penalty0 1933--1946, 2016.

\bibitem[Zhang(2016)]{HuiZhang16a}
Hui Zhang.
\newblock Characterization of linear convergence of gradient descent.
\newblock \emph{arXiv:1606.00269}, 2016.

\bibitem[Zhang and Cheng(2015)]{HuiZhang:2015}
Hui Zhang and Lizhi Cheng.
\newblock Restricted strong convexity and its applications to convergence
  analysis of gradient-type methods in convex optimization.
\newblock \emph{Optimization Letter}, 9:\penalty0 961--979, 2015.

\bibitem[Zhou and So(2017)]{ZhouSo15}
Zirui Zhou and Anthony Man-Cho So.
\newblock A unified approach to error bounds for structured convex optimization
  problems.
\newblock \emph{Mathematical Programming}, 165\penalty0 (2):\penalty0 689--728,
  Oct 2017.

\bibitem[Zhou et~al.(2015)Zhou, Zhang, and So]{DBLP:conf/icml/ZhouZS15}
Zirui Zhou, Qi~Zhang, and Anthony~Man{-}Cho So.
\newblock L1p-norm regularization: Error bounds and convergence rate analysis
  of first-order methods.
\newblock In \emph{Proceedings of the 32nd International Conference on Machine
  Learning, (ICML)}, pages 1501--1510, 2015.

\bibitem[Zinkevich(2003)]{zinkevich-2003-online}
Martin Zinkevich.
\newblock Online convex programming and generalized infinitesimal gradient
  ascent.
\newblock In \emph{Proceedings of the International Conference on Machine
  Learning (ICML)}, pages 928--936, 2003.

\end{thebibliography}

\end{document}